\documentclass[review,hidelinks,onefignum]{siamart220329}
\usepackage{lineno}
\usepackage{subfigure}  %
\usepackage{float}
\usepackage{graphicx}
\usepackage{multirow}
\usepackage{multicol}
\usepackage{amssymb}

\usepackage{threeparttable}
\def\max{\mbox{max}\,}

\newtheorem{example}{Example}[section]

\usepackage{lipsum}
\usepackage{amsfonts}
\usepackage{graphicx}
\usepackage{epstopdf}
\usepackage{algorithmic}
\ifpdf
  \DeclareGraphicsExtensions{.eps,.pdf,.png,.jpg}
\else
  \DeclareGraphicsExtensions{.eps}
\fi

\newsiamremark{remark}{Remark}
\newsiamremark{hypothesis}{Hypothesis}
\crefname{hypothesis}{Hypothesis}{Hypotheses}
\newsiamthm{claim}{Claim}

\title{Randomized Quaternion UTV Decomposition and Randomized Quaternion Tensor UTV Decomposition\thanks{Submitted to the editors DATE.
\funding{This work was funded by the University of Macau (MYRG2022-00108-FST, MYRG-CRG2022-00010-ICMS).}}}

\author{Liqiao Yang\thanks{Research Institute for Digital Economy and Interdisciplinary Sciences, Southwestern University of Finance and Economics, Chengdu, Sichuan, 611130, China
  (\email{liqiaoyoung@163.com}).}
\and Jifei Miao\thanks{
	School of Mathematics and Statistics, Yunnan University, Kunming, Yunnan, 650091, China
	(\email{jifmiao@163.com},  \email{yanlnzhang@163.com}, \email{kikou@umac.mo}).}
\and Tai-Xiang Jiang \thanks{School of Computing and Artificial Intelligence, Southwestern University of Finance and Economics, Chengdu, Sichuan, 611130, China
  (\email{taixiangjiang@gmail.com}).}
\and Yanlin Zhang  \footnotemark[5]
\and Kit Ian Kou 
\thanks{Department of  Mathematics, Faculty of Science and Technology, University of Macau, Macau 999078, China 
	(\email{yanlnzhang@163.com}, \email{kikou@umac.mo}).}
}

\usepackage{amsopn}

\ifpdf
\hypersetup{
  pdftitle={Randomized Quaternion Tensor UTV Decomposition},
  pdfauthor={L. Yang, J. Miao, T. Jiang, Y. Zhang, K. I. Kou}
}
\fi

\begin{document}
\nolinenumbers
\maketitle

\begin{abstract}
In this paper, the quaternion matrix  UTV (QUTV) decomposition and quaternion tensor UTV (QTUTV) decomposition are proposed. To begin, the terms QUTV and QTUTV are defined, followed by the algorithms. Subsequently, by employing random sampling from the quaternion normal distribution, randomized QUTV and randomized QTUTV are generated to provide enhanced algorithmic efficiency. These techniques produce decompositions that are straightforward to understand and require minimal cost. Furthermore, theoretical analysis is discussed. Specifically, the upper bounds for approximating QUTV on the rank-$K$ and QTUTV on the TQt-rank $K$ errors are provided, followed by deterministic error bounds and average-case error bounds for the randomized situations, which demonstrate the correlation between the accuracy of the low-rank approximation and the singular values. Finally, numerous numerical experiments are presented to verify that the proposed algorithms work more efficiently and with similar relative errors compared to other comparable decomposition methods. For the novel decompositions, the theory analysis offers a solid theoretical basis and the experiments show significant potential for the associated processing tasks of color images and color videos. 
\end{abstract}



\section{Introduction}
\label{sec:intro}

As the concept of quaternion is proposed by Hamilton in 1843 \cite{doi:10.1080/14786444408644923}, quaternions and quaternion matrices are gradually applied to various fields. Such as Bayro \textit{et al.} pioneered in the application of quaternion to robotics \cite{DBLP:conf/icpr/DaniilidisB96, DBLP:journals/tii/LiuWZ15, DBLP:journals/firai/SunHFXYMW23, DBLP:journals/ral/WuWZY24} and neurocomputing \cite{DBLP:conf/icnn/Bayro-Corrochano96, DBLP:journals/ijon/LiuWCL24, DBLP:journals/fss/LiC24, DBLP:journals/tfs/ChenZYSLW24}, etc. Besides, computer vision based on quaternion representation also has  received widespread attention, such as image restoration or segmentation \cite{DBLP:journals/pr/MiaoKL20, DBLP:journals/siamis/ChenN22, DBLP:journals/jscic/WuMLZZ22}, face recognition \cite{DBLP:journals/jscic/KeMJXL23, DBLP:journals/tip/LiuKMC23, DBLP:journals/sigpro/LingLYJ22}, and watermarking \cite{DBLP:journals/sigpro/ChenJPPZ21, DBLP:journals/sigpro/ZhangDLSL23, DBLP:journals/tip/ChenJPP23}, etc. For the approaches based on quaternion representation, a vital problem is to find a cost-saving expression. Since many data matrices are low-rank inherently, there are various approaches to get the low-rank expression to approximate the original data. 

One of the most straightforward approaches for optimizing the description of low rankness in a quaternion matrix is to utilize Quaternion Singular Value Decomposition (QSVD).  Similar to the real cases, the rank of quaternion matrix is defined as the number of non-zero singular values, and the quaternion nuclear norm (QNN) as an efficient replacement of rank has been well researched. In specific, the weighted QNN \cite{DBLP:journals/pr/HuangLLWZ22}, the weighted Schatten p-norm \cite{DBLP:journals/tip/ChenXZ20}, etc, have been proposed to optimize QNN. Another low rank approximation method is low-rank factorization. In \cite{DBLP:journals/tsp/MiaoK20, DBLP:journals/siamis/PanN23}, the target quaternion matrix is factorized to the product of two small size quaternion matrices such that the low rankness is depicted by this two smaller quaternion matrices. Besides, some classic factorization also have been developed in the quaternion domain. Such as quaternion QR (QQR) decomposition \cite{DBLP:journals/sigpro/ChenJPPZ21}, quaternion LU decomposition \cite{DBLP:journals/cphysics/WangM13, DBLP:journals/na/ZhangWLZ24}, etc. Further, to optimize the expensive process of obtaining QSVD and improve the efficiency of computation, randomized algorithms is introduced to QSVD \cite{DBLP:journals/siamsc/LiuLJ22}. Hereafter, the randomized quaternion QR with Pivoting and Least Squares (QLP) decomposition \cite{DBLP:journals/jscic/RenMLB22} and quaternion-based CUR decomposition \cite{doi:10.1080/01630563.2024.2318602} are proposed to enrich the numerical algorithms of quaternion matrices.

In the case of a quaternion tensor, analogous to real tensors, the definition of rank is not unique. Therefore, various factorizations of the quaternion tensor would result in different definitions of rank. The concept of Tucker rank for a quaternion tensor is introduced in \cite{DBLP:journals/pr/MiaoKL20} for the purpose of handling color images and color videos. In this sense, a color video is seen as a pure quaternion tensor of third order. In addition, a flexible transform-based tensor product called the $\star_{QT}$-product is introduced, along with its associated  Quaternion Tensor Singular Value Decomposition  known as TQt-SVD, as specified in the paper by Miao and Kou \cite{DBLP:journals/sigpro/MiaoK23}. The efficacy of TQt-rank is further demonstrated by the utilization of low-rank approximation in color videos. In addition, the quaternion tensor train is also introduced in the reference \cite{DBLP:journals/kbs/MiaoKYC24}.

While there has been extensive research on approximating quaternion matrices and quaternion tensors, the majority of studies have mostly concentrated on optimizing low-rank descriptions rather than exploring more efficient factorization approaches. In this paper, the UTV decomposition is specifically developed for quaternion matrices and quaternion tensors to address the existing gap. A UTV decomposition of an $m \times n$ matrix is the result of multiplying an orthogonal matrix $\mathbf{U}$, a middle triangular matrix $\mathbf{T}$, and another orthogonal matrix $\mathbf{V}$ together \cite{DBLP:journals/tsp/Stewart92, DBLP:journals/siammax/Stewart93, DBLP:journals/na/FierroH97}. Hence, the UTV decomposition is classified into two types based on whether the middle triangular matrix is an upper triangular matrix (referred to as URV decomposition) or a lower triangular matrix (referred to as ULV decomposition).  Specifically, when considering the URV decomposition of a real matrix $\mathbf{A}$, the primary computational procedure involves doing the QR with column pivoting (QRCP) decomposition twice:  $\mathbf{A}^H=\mathbf{Q}_1\mathbf{R}_1\mathbf{P}_1^T$  and $(\mathbf{R}_1\mathbf{P}_1^T)^H=\mathbf{Q}_2\mathbf{R}_2\mathbf{P}_2^T$, where $\mathbf{P}_i$ is permutation matrix. Such that $\mathbf{A}=\mathbf{U}\mathbf{R}\mathbf{V}^H$, where $\mathbf{U}=\mathbf{Q}_2$, $\mathbf{R}=\mathbf{R}_2$, and $\mathbf{V}=\mathbf{Q}_1\mathbf{P}_2$ \cite{DBLP:journals/siamsc/Stewart99}. The developed QUTV can be seen as an evolution from UTV in the real domain to UTV in the quaternion domain. Moreover, utilizing the $\star_{QT}$-product, the QUTV builds up to quaternion tensors. The $\star_{QT}$-product is defined via a flexible transform. By employing the quaternion Discrete Fourier Transformation (QDFT) and the quaternion Discrete Cosine Transformation (QDCT), among others. The QTUTV algorithm involves reshaping the $N$th-order quaternion tensor into a third-order quaternion tensor. Then, the QRCP is applied twice to each frontal slice of this modified tensor in the altered domain.

Because the created decomposition acts on the entire original quaternion, the calculation costs are substantial. To achieve more efficient computing, the randomized technique is applied to QUTV and QTUTV. This is prompted by the excellent performance in randomized QSVD \cite{DBLP:journals/siamsc/LiuLJ22}. Further, we apply this randomized technique to quaternion tensors. This indicates the initial application of the randomized technique to a quaternion tensor, as far as we are aware.
Moreover, this study analyzes error boundaries and average-case error, building on the work of randomized QSVD in \cite{DBLP:journals/siamsc/LiuLJ22} and the real URV decomposition in \cite{DBLP:journals/tsp/KalooraziL18}.

To conclude, the primary contribution of this study can be summarized as follows.
\begin{itemize}
	\item The UTV decomposition is proposed for quaternion matrix by employing quaternion QRCP (QQRCP). Further,  based on the $\star_{QT}$-product, the QUTV decomposition is generated to  quaternion tensor, named QTUTV. 
	
	\item In order to enhance the efficiency of algorithms and reduce computational expenses, the QUTV and QTUTV decomposition methods have been supplemented with a randomized approach. In addition, the upper bounds for the proposed methods are given, and the average error bounds on randomized QUTV and randomized QTUTV are established. 
	
	\item The numerical results demonstrate that the created QUTV achieves comparable relative errors to QSVD while using less time. Similarly, QTUTV can be likened to QTSVD. The efficiency of randomized QUTV and randomized QTUTV is proved in the randomized cases, with an acceptable relative error compared to the randomized QSVD and randomized QTSVD methods, respectively.
	
\end{itemize}

The rest of this paper consists of six more sections.  Section \ref{sec:1} provides an overview of the relevant notations and fundamental concepts relating to quaternion matrices and quaternion tensors. In Section \ref{sec:2}, new quaternion-based QUTV and QTUTV methods are introduced and described in detail. In Section \ref{sec:3}, randomized QURV and randomized QTURV are introduced, and theoretical analysis is provided in Section \ref{TA}. Section \ref{E} demonstrates the effectiveness of the developed methods by comparing numerical experiments with other relevant algorithms. The conclusion is stated in Section \ref{C}. Relevant proofs are available in the Appendix.

\section{Notations and preliminaries}\label{sec:1}

\subsection{Notations}
The symbols a, $\textbf{a}$, $\textbf{A}$, and $\mathcal{A}$ represent scalar, vector, matrix, and tensor quantities in the real number field $\mathbb{R}$, respectively. The dot notation is employed to represent the scalar, vector, matrix, and tensor variables in the quaternion domain $\mathbb{H}$. Specifically, $\dot{a}$ denotes a scalar, $\dot{\textbf{a}}$ represents a vector, $\dot{\textbf{A}}$ signifies a matrix, and $\dot{\mathcal{A}}$ indicates a tensor. In addition, the complex space is represented by the symbol $\mathbb{C}$. The symbol $\textbf{I}_{r}$ denotes the identity matrix of size $r\times r$, while $\textbf{O}_{r_1\times r_2}$ represents the matrix of zeros with dimensions $r_1\times r_2$.  The symbols $\mathbf{(\cdot)}^T$, $\mathbf{(\cdot)}^H$, and $\mathbf{(\cdot)}^{-1}$ denote the transpose, conjugate transpose, and inverse  respectively. The $k$th frontal slice of an $N$th-order quaternion tensor $\dot{\mathcal{A}}\in\mathbb{H}^{I_1\times I_2\times \cdots \times I_N}$ is denoted by $\dot{\mathcal{A}}(:,:,i_3, \cdots, i_N)$, and the mode-\textit{k} unfolding is denoted by $\dot{\mathcal{A}}_{(k)}$. The product of the \textit{k}-mode is represented by the symbol $\times_k$. The symbols $\|\cdot\|_F$ and $\|\cdot\|_*$ represent the Frobenius norm and nuclear norm of a given quantity, respectively.  The inner product of  matrices $\cdot_1$ and $\cdot_2$ is defined as $\langle\cdot_1, \cdot_2\rangle\triangleq \text{tr}(\cdot_1^H\cdot_2)$, and $tr(\cdot)$ is the trace function.

\subsection{Preliminary of quaternion and quaternion matrix}
A quaternion number $\dot{a}\in\mathbb{H}$ can be depicted by $
\dot{a}=a_0+a_1\emph{i}+a_2\emph{j}+a_3\emph{k}$, where $a_n\in\mathbb{R}$ $(n=0,1,2,3)$, and \emph{i, j, k} are three imaginary units which satisfy:
\begin{equation*}\label{m4}
	\begin{cases}
		\emph{i}^2= \emph{j}^2 =\emph{k}^2= \emph{i}\emph{j}\emph{k}=-1\\
		\emph{i}\emph{j}=-\emph{j}\emph{i} = \emph{k},    \emph{j}\emph{k}=-\emph{k}\emph{j} = \emph{i},  \emph{k}\emph{i}=-\emph{i}\emph{k} = \emph{j}.
	\end{cases}
\end{equation*}
$\mathfrak{R}(\dot{a}) \triangleq a_0$ and $\mathfrak{I}(\dot{a}) \triangleq a_1\emph{i}+a_2\emph{j}+a_3\emph{k}$ denote the real part and imaginary part of  $\dot{a}$, respectively, so $\dot{a}=\mathfrak{R}(\dot{a})+\mathfrak{I}(\dot{a})$. When $a_0 = 0$, $\dot{a}$ is  a pure quaternion.  The conjugate of $\dot{a}$  is defined as
\begin{equation*}
\dot{a}^{*} = a_0-a_1\emph{i}-a_2\emph{j}-a_3\emph{k}, 
\end{equation*}
and the modulus of $\dot{a}$ is 
\begin{equation*}
|\dot{a}|= \sqrt{\dot{a}\dot{a}^{*}}=\sqrt{a_0^2+a_1^2+a_2^2+a_3^2}.
\end{equation*}

It is essential to acknowledge that multiplication in the quaternion domain does not adhere to the commutative property, \textit{i.e.} $\dot{a}\dot{b} \neq \dot{b}\dot{a}$.

A quaternion matrix $\dot{\textbf{A}}=(\dot{a}_{ij})\in\mathbb{H}^{M \times N}$, where $\dot{\textbf{A}}=\textbf{A}_0+\textbf{A}_1\emph{i}+\textbf{A}_2\emph{j}+\textbf{A}_3\emph{k}$, and $\textbf{A}_n\in\mathbb{R}^{M \times N}$ $(n=0,1,2,3)$. The conjugate transpose of  $\dot{\textbf{A}}$ is
\begin{equation*}
\dot{\textbf{A}}^H=\textbf{A}_0^T-\textbf{A}_1^T\emph{i}-\textbf{A}_2^T\emph{j}-\textbf{A}_3^T\emph{k}.
\end{equation*} The Frobenius norm is defined as \begin{equation*}
\parallel\dot{\textbf{A}}\parallel_F = \sqrt{\sum_{i=1}^{M}\sum_{j=1}^{N}|\dot{a}_{ij}|^2}=\sqrt{tr(\dot{\textbf{A}}^H\dot{\textbf{A}})}. \end{equation*} The spectral norm is defined as \begin{equation*}\parallel\dot{\mathbf{A}}\parallel_2 = \mathop{\max}\limits_{\dot{\mathbf{x}}\in\mathbb{H^N}, \parallel\dot{\mathbf{x}}\parallel_2=1}\parallel \dot{\mathbf{A}}\dot{\mathbf{x}}\parallel_2, \end{equation*}
where the 2-norm is $\parallel\dot{\mathbf{A}}\dot{\mathbf{x}}\parallel_2 =(\sum_{i=1}^{N}\mid\dot{\mathbf{A}}\dot{\mathbf{x_i}}\mid^2)^{\frac{1}{2}}.$  

\begin{definition}(\textbf{The rank of quaternion matrix} \cite{zhang1997quaternions}): For a quaternion matrix $\dot{\mathbf{A}}=(\dot{a}_{ij})\in\mathbb{H}^{M \times N}$, the  maximum number of right (left) linearly independent columns (rows)  is defined as the rank of  $\dot{\mathbf{A}}$.
\end{definition}

\begin{theorem}(\textbf{QSVD} \cite{zhang1997quaternions}):
	\label{theorem1}
	Let a quaternion matrix $\dot{\mathbf{A}}\in\mathbb{H}^{M \times N}$ be of rank $r$. There exist two unitary quaternion matrices $\dot{\mathbf{U}}\in\mathbb{H}^{M \times M}$
	and $\dot{\mathbf{V}}\in\mathbb{H}^{N \times N}$ such that
	\begin{equation}
		\dot{\mathbf{A}}=\dot{\mathbf{U}}
		\left( \begin{array}{cc}
			\mathbf{\Sigma}_r & \mathbf{0}  \\
			\mathbf{0} & \mathbf{0}\\
		\end{array}
		\right )\dot{\mathbf{V}}^H= \dot{\mathbf{U}}\mathbf{\Lambda}\dot{\mathbf{V}}^H,
	\end{equation}
	where $\mathbf{\Sigma}_r=diag({\sigma_1,\cdots, \sigma_r})\in\mathbb{R}^{r\times r}$, and all singular values $\sigma_i>0,  i=1,\cdots,r$. 
\end{theorem}

\subsection{Preliminary of quaternion tensor}
\begin{definition}[\textbf{Quaternion tensor}\cite{DBLP:journals/pr/MiaoKL20}]
	\label{def3}
	A multi-dimensional array or an Nth-order tensor is referred to as a quaternion tensor when its elements are quaternions. Specifically, quaternion tensor $\dot{\mathcal{A}}=(\dot{a}_{i_1i_2\ldots i_N})\in\mathbb{H}^{I_1\times I_2\times \cdots \times I_N}$ is formulated as $\dot{\mathcal{A}}=\mathcal{A}_0+\mathcal{A}_1i+\mathcal{A}_2j+\mathcal{A}_3k$, where $\mathcal{A}_n\in\mathbb{R}^{I_1\times I_2\times \cdots \times I_N}$ $(n=0,1,2,3)$ are real tensors. 
\end{definition}

\begin{definition}[\textbf{$\star_{QT}$-product} \cite{DBLP:journals/sigpro/MiaoK23}]
	\label{def4}
	Given two Nth-order ($N\geq3$) quaternion tensors $\dot{\mathcal{A}}\in\mathbb{H}^{I_1\times l\times I_3\times \ldots \times I_N}$, $\dot{\mathcal{B}}\in\mathbb{H}^{l\times I_2\times I_3\times \ldots \times I_N}$ and $N-2$ invertible quaternion matrices $\dot{\mathbf{Q}}_3\in\mathbb{H}^{I_3\times I_3}, \cdots,  \dot{\mathbf{Q}}_N\in\mathbb{H}^{I_N\times I_N}$, the $\star_{QT}$-product is defined as 
	\begin{equation}
		\dot{\mathcal{T}}=\dot{\mathcal{A}}\star_{QT}\dot{\mathcal{B}}=(\hat{\dot{\mathcal{A}}}\star_{QF}\hat{\dot{\mathcal{B}}})\times_3\dot{\mathbf{Q}}_3^{-1}\times_4 \cdots \times_N\dot{\mathbf{Q}}_N ^{-1},
	\end{equation}	
	where $\hat{\dot{\mathcal{A}}}=\dot{\mathcal{A}}\times_3\dot{\mathbf{Q}}_3\times_4 \cdots \times_N\dot{\mathbf{Q}}_N $ and $\hat{\dot{\mathcal{B}}}=\dot{\mathcal{B}}\times_3\dot{\mathbf{Q}}_3\times_4 \cdots \times_N\dot{\mathbf{Q}}_N $. The $\star_{QF}$-product is the quaternion facewise produce, \textit{i.e.}, $\dot{\mathcal{F}}=\dot{\mathcal{A}}\star_{QF}\dot{\mathcal{B}}$ such that the frontal slice of $\dot{\mathcal{F}}$ satisfies $\dot{\mathcal{F}}(:, :, i_3, \cdots, i_N)=\dot{\mathcal{A}}(:, :, i_3, \cdots, i_N)\dot{\mathcal{B}}(:, :, i_3, \cdots, i_N)$.
\end{definition}
\begin{remark}
	Based on $\star_{QT}$-product, the conjugate transpose of $\dot{\mathcal{A}}$ is denoted by $\dot{\mathcal{A}}^H$ and satisfies $\hat{\dot{\mathcal{A}}}^H(:, :, i_3, \cdots, i_N)=(\hat{\dot{\mathcal{A}}}(:, :, i_3, \cdots, i_N))^H$. The identity quaternion tensor $\dot{\mathcal{I}}\in\mathbb{H}^{l\times l\times \cdots \times I_N}$ has the property that $\hat{\dot{\mathcal{I}}}(:, :, i_3, \cdots, i_N)=\dot{\mathbf{I}}\in\mathbb{H}^{l\times l}$. If a quaternion tensor $\dot{\mathcal{U}}\in\mathbb{H}^{l\times l\times \cdots \times I_N}$ satisfies $\dot{\mathcal{U}}\star_{QT}\dot{\mathcal{U}}^H=\dot{\mathcal{I}}=\dot{\mathcal{U}}^H\star_{QT}\dot{\mathcal{U}}$,  $\dot{\mathcal{U}}$ is a unitary quaternion tensor.
\end{remark}
\begin{theorem}[\textbf{TQt-SVD }\cite{DBLP:journals/sigpro/MiaoK23}]
	\label{th2}
	Let $\dot{\mathcal{T}}\in\mathbb{H}^{I_1\times I_2\times \cdots \times I_N} (N\geq3)$. There exist two unitary quaternion tensors $\dot{\mathcal{U}}\in\mathbb{H}^{I_1\times I_1\times \cdots \times I_N}$
	and $\dot{\mathcal{V}}\in\mathbb{H}^{I_2\times I_2\times \cdots \times I_N}$. Such that
	\begin{equation*}
		\dot{\mathcal{T}}= \dot{\mathcal{U}}\star_{QT}\dot{\mathcal{D}}\star_{QT}\dot{\mathcal{V}}^H,
	\end{equation*}
	where the tensor $\dot{\mathcal{D}}\in\mathbb{H}^{I_1\times I_2\times \cdots \times I_N}$ is an f-diagonal quaternion tensor, which means that only the diagonal components of its frontal slices have non-zero values. 
\end{theorem}

\begin{definition}[\textbf{TQt-rank} \cite{DBLP:journals/sigpro/MiaoK23}] Let  $\dot{\mathcal{T}}\in\mathbb{H}^{I_1\times I_2\times \cdots \times I_N} (N\geq3)$, and the corresponding TQt-SVD is $\dot{\mathcal{T}}= \dot{\mathcal{U}}\star_{QT}\dot{\mathcal{D}}\star_{QT}\dot{\mathcal{V}}^H$. The TQt-rank of $\dot{\mathcal{T}}$ is defined as  the number of nonzero tubes in $\dot{\mathcal{D}}$, \textit{i.e.}, $rank_{TQt}(\dot{\mathcal{T}})= \#\{k\mid\parallel\dot{\mathcal{D}}(k, k,:, \cdots, :)\parallel_F>0\}, k\in[K], K=\min\{I_1,I_2\}$. Moreover, the \textit{k}th singular value of $\dot{\mathcal{T}}$ is defined as $\sigma_{k}(\dot{\mathcal{T}})=\parallel\dot{\mathcal{D}}(k, k,:, \cdots, :)\parallel_F$.
\end{definition}

Additional details on quaternions, quaternion matrices, and quaternion tensors can be found in the following references: \cite{DBLP:journals/pr/MiaoKL20, DBLP:journals/sigpro/MiaoK23, zhang1997quaternions}.

\section{Quaternion UTV decompositions}\label{sec:2}
Firstly, this section introduces the quaternion matrix UTV decomposition and presents the QURV algorithm and the QULV algorithm. Furthermore, a novel decomposition called the quaternion tensor UTV decomposition is introduced, along with the provided QTURV decomposition algorithm.
\subsection{Quaternion matrix UTV decomposition}
The QUTV decomposition of a quaternion matrix can be classified into two classes in analogy to the real cases.
For a quaternion matrix $\dot{\mathbf{A}}\in\mathbb{H}^{M \times N}$, the QUTV decomposition is $\dot{\mathbf{A}}=\dot{\mathbf{U}}\dot{\mathbf{T}}\dot{\mathbf{V}}^H$, where $\dot{\mathbf{U}}\in\mathbb{H}^{M \times M}$, $\dot{\mathbf{V}}\in\mathbb{H}^{N \times N}$ are two unitary quaternion matrices, and $\dot{\mathbf{T}}\in\mathbb{H}^{M \times N}$ is a triangular quaternion matrix. The QURV decomposition is formulated as follows when $\dot{\mathbf{T}}$ is upper triangular.

\begin{equation}
	\dot{\mathbf{A}}=\dot{\mathbf{U}}
	\left( \begin{array}{cc}
		\dot{\mathbf{T}}_{11} & \dot{\mathbf{T}}_{12} \\
		\mathbf{0}_{(M-K)\times (K)} &\dot{\mathbf{T}}_{22}\\
	\end{array}
	\right )\dot{\mathbf{V}}^H,
\end{equation}
where $\dot{\mathbf{T}}_{11}\in \mathbb{H}^{K \times K}$, $\dot{\mathbf{T}}_{12}\in \mathbb{H}^{K \times (N-K)}$, and $\dot{\mathbf{T}}_{22}\in \mathbb{H}^{(M-K) \times (N-K)}$.
The QULV decomposition is formulated as follows when $\dot{\mathbf{T}}$ is lower triangular.

\begin{equation}
	\dot{\mathbf{A}}=\dot{\mathbf{U}}
	\left( \begin{array}{cc}
		\dot{\mathbf{T}}_{11} &\mathbf{0}_{K\times (N-K)} \\
		\dot{\mathbf{T}}_{21} 	 &\dot{\mathbf{T}}_{22}\\
	\end{array}
	\right )\dot{\mathbf{V}}^H,
\end{equation}
where $\dot{\mathbf{T}}_{11}\in \mathbb{H}^{K \times K}$, $\dot{\mathbf{T}}_{21}\in \mathbb{H}^{(M-K) \times K}$, and $\dot{\mathbf{T}}_{22}\in \mathbb{H}^{(M-K) \times (N-K)}$.

Basing on the above situations, when the middle quaternion matrix is diagonal and the elements are real nonnegative, the decomposition is referred as QSVD. 
\begin{equation}
	\dot{\mathbf{A}}=\dot{\mathbf{U}}
	\left( \begin{array}{cc}
		\dot{\mathbf{T}}_{11} &\mathbf{0}_{K\times (N-K)} \\
		\mathbf{0}_{ (M-K)\times K} 	 &\dot{\mathbf{T}}_{22}\\
	\end{array}
	\right)\dot{\mathbf{V}}^H.
\end{equation}
In this case, the diagonal elements in $\dot{\mathbf{T}}$ are singular values of  $\dot{\mathbf{A}}$, and $\dot{\mathbf{A}}_K=\dot{\mathbf{U}}(:,1:K)\dot{\mathbf{T}}_{11}\dot{\mathbf{V}}(:,1:K)^H$ provides the best rank-K approximation for $\dot{\mathbf{A}}$ \cite{eckart1936approximation}.

A beneficial method for computing URV decomposition can be achieved by utilizing QRCP, as indicated in Section \ref{sec:1}. Further, Algorithm \ref{a1.0} provides a clear overview of the QURV decomposition process through the use of the QQRCP. Similarly, Algorithm \ref{a1.1} provides a description of the QULV decomposition.
\begin{algorithm}[htbp]
	\caption{The QURV decomposition}
	\label{a1.0}
	\begin{algorithmic}[1]
		\REQUIRE   Quaternion matrix data $\dot{\mathbf{A}}\in\mathbb{H}^{M\times N}$.
		\ENSURE    Two unitary quaternion matrices: $\dot{\mathbf{U}}\in\mathbb{H}^{M \times M}$, $\dot{\mathbf{V}}\in\mathbb{H}^{N \times N}$  and an upper triangular quaternion matrix $\dot{\mathbf{R}}\in\mathbb{H}^{M \times N}$ such that $\dot{\mathbf{A}}=\dot{\mathbf{U}}\dot{\mathbf{R}}\dot{\mathbf{V}}^H$.
		\STATE Compute the QQRCP of $\dot{\mathbf{A}}^H$ as $\dot{\mathbf{A}}^H=\dot{\mathbf{Q}}_1\dot{\mathbf{R}}_1\mathbf{P}_1^T$.
		\STATE Compute the QQRCP of $(\dot{\mathbf{R}}_1\mathbf{P}_1^T)^H$ as $(\dot{\mathbf{R}}_1\mathbf{P}_1^T)^H=\dot{\mathbf{Q}}_2\dot{\mathbf{R}}_2\mathbf{P}_2^T$.
		\STATE  $\dot{\mathbf{U}}=\dot{\mathbf{Q}}_2$, $\dot{\mathbf{R}}=\dot{\mathbf{R}}_2$, and $\dot{\mathbf{V}}=\dot{\mathbf{Q}}_1\mathbf{P}_2$.		
	\end{algorithmic}
\end{algorithm}

\begin{algorithm}[htbp]
	\caption{The QULV decomposition}
	\label{a1.1}
	\begin{algorithmic}[1]
		\REQUIRE   Quaternion matrix data $\dot{\mathbf{A}}\in\mathbb{H}^{M\times N}$.
		\ENSURE    Two unitary quaternion matrices: $\dot{\mathbf{U}}\in\mathbb{H}^{M \times M}$, $\dot{\mathbf{V}}\in\mathbb{H}^{N \times N}$  and a lower triangular quaternion matrix $\dot{\mathbf{L}}\in\mathbb{H}^{M \times N}$ such that $\dot{\mathbf{A}}=\dot{\mathbf{U}}\dot{\mathbf{L}}\dot{\mathbf{V}}^H$.
		\STATE Compute the QQRCP of $\dot{\mathbf{A}}$ as $\dot{\mathbf{A}}=\dot{\mathbf{Q}}_1\dot{\mathbf{R}}_1\mathbf{P}_1^T$.
		\STATE Compute the QQRCP of $(\dot{\mathbf{R}}_1\mathbf{P}_1^T)^H$ as $(\dot{\mathbf{R}}_1\mathbf{P}_1^T)^H=\dot{\mathbf{Q}}_2\dot{\mathbf{R}}_2\mathbf{P}_2^T$.
		\STATE  $\dot{\mathbf{U}}=\dot{\mathbf{Q}}_1\mathbf{P}_2$, $\dot{\mathbf{L}}=\dot{\mathbf{R}}_2^H$, and $\dot{\mathbf{V}}=\dot{\mathbf{Q}}_2$.	
	\end{algorithmic}
\end{algorithm}

\subsection{Quaternion tensor UTV decomposition}
The following QTUTV decomposition is proposed for the $N$th-order quaternion tensor, which is based on the $\star_{QT}$-product.
\begin{theorem}[\textbf{QTUTV }]
	\label{th3}
	Let $\dot{\mathcal{T}}\in\mathbb{H}^{I_1\times I_2\times \cdots \times I_N} (N\geq3)$.  $\dot{\mathcal{T}}$ can be decomposed to
	\begin{equation*}
		\dot{\mathcal{A}}= \dot{\mathcal{U}}\star_{QT}\dot{\mathcal{T}}\star_{QT}\dot{\mathcal{V}}^H,
	\end{equation*}
	where $\dot{\mathcal{U}}\in\mathbb{H}^{I_1\times I_1\times \cdots \times I_N}$,
	and $\dot{\mathcal{V}}\in\mathbb{H}^{I_2\times I_2\times \cdots \times I_N}$ are unitary quaternion tensors, and $\dot{\mathcal{T}}\in\mathbb{H}^{I_1\times I_2\times \cdots \times I_N}$ is a frontal triangular quaternion tensor, indicating that the frontal slices of  $\dot{\mathcal{T}}$ are triangular.  
\end{theorem}

\begin{remark}
	When the frontal triangular quaternion tensor $\dot{\mathcal{T}}$ is f-upper triangular (meaning that the frontal slices are of upper triangular quaternion matrices), the QTUTV decomposition is equivalent to QTURV decomposition. Similarly, when the triangular quaternion tensor $\dot{\mathcal{T}}$ is f-lower triangular (meaning that frontal slices are of lower triangular quaternion matrices), the QTUTV decomposition is equivalent to QTULV decomposition. Moreover, if the triangular quaternion tensor $\dot{\mathcal{T}}$ satisfies the condition of being f-diagonal triangular (i.e., the frontal slices consist of diagonal quaternion matrices), then the QTUTV decomposition can be referred to as the QTSVD decomposition. 
\end{remark}

Algorithm \ref{a1.2} summarizes the QTURV decomposition using QQRCP decomposition for each frontal slice of the quaternion tensor in the transform domain. The algorithmic procedure of QTULV is analogous and will not be reiterated here.
\begin{algorithm}[htbp]
	\caption{The QTURV decomposition}
	\label{a1.2}
	\begin{algorithmic}[1]
		\REQUIRE   Quaternion tensor data $\dot{\mathcal{A}}\in\mathbb{H}^{{I_1\times I_2\times \cdots \times I_N}}(N\geq3)$, and invertible quaternion matrices $\dot{\mathbf{Q}}_3\in\mathbb{H}^{I_3\times I_3}, \cdots,  \dot{\mathbf{Q}}_N\in\mathbb{H}^{I_N\times I_N}$.
		\ENSURE    Two unitary quaternion tensors:
		$\dot{\mathcal{U}}\in\mathbb{H}^{I_1\times I_1\times \cdots \times I_N}$,
		$\dot{\mathcal{V}}\in\mathbb{H}^{I_2\times I_2\times \cdots \times I_N}$, and a f-upper quaternion tensor $\dot{\mathcal{R}}\in\mathbb{H}^{I_1\times I_2\times \cdots \times I_N}$ such that $\dot{\mathcal{A}}= \dot{\mathcal{U}}\star_{QT}\dot{\mathcal{T}}\star_{QT}\dot{\mathcal{V}}^H$.
		\STATE Convert $\dot{\mathcal{A}}$ to transform domain: $\hat{\dot{\mathcal{A}}}=\dot{\mathcal{A}}\times_3\dot{\mathbf{Q}}_3\times_4 \cdots \times_N\dot{\mathbf{Q}}_N $.
		\STATE Concatenate frontal slices along the third dimension:\\ $\hat{\dot{\mathcal{A}}}=\text{reshape}(\hat{\dot{\mathcal{A}}},[I_1,  I_2, I_3I_4\cdots I_N])$.
		\STATE \textbf{for} k=1:$I_3I_4\cdots I_N$ \textbf{do}
		\STATE \quad  Compute the QQRCP of $\hat{\dot{\mathcal{A}}}(:, :, k)^H$:\\ \quad $\hat{\dot{\mathcal{A}}}(:, :, k)^H=\hat{\dot{\mathcal{Q}}}_1(:, :, k)\hat{\dot{\mathcal{R}}}_1(:, :, k)\mathcal{P}_1(:, :, k)^T$.
		\STATE \quad Compute the QQRCP of $(\hat{\dot{\mathcal{R}}}_1(:, :, k)\mathcal{P}_1(:, :, k)^T)^H$:\\ \quad $(\hat{\dot{\mathcal{R}}}_1(:, :, k)\mathcal{P}_1(:, :, k)^T)^H=\hat{\dot{\mathcal{Q}}}_2(:, :, k)\hat{\dot{\mathcal{R}}}_2(:, :, k)\mathcal{P}_2(:, :, k)^T$.
		\STATE \textbf{end for}
		\STATE Reshape into $N$th-order quaternion tensors:\\ 
		$\hat{\dot{\mathcal{U}}}=\text{reshape}(\hat{\dot{\mathcal{Q}}}_2, [I_1,  I_1, I_3,\cdots, I_N])$;\\ $\hat{\dot{\mathcal{R}}}=\text{reshape}(\hat{\dot{\mathcal{R}}}_2, [I_1,  I_2, I_3, \cdots, I_N])$;\\ $\hat{\dot{\mathcal{V}}}=\text{reshape}(\hat{\dot{\mathcal{Q}}}_1\star_{QF}\mathcal{P}_2, [I_2,  I_2, I_3,\cdots, I_N])$.		
		\STATE Return to the original domain:\\ $\dot{\mathcal{U}}=\hat{\dot{\mathcal{U}}}\times_3\dot{\mathbf{Q}}_3^{-1}\times_4\cdots \times_N\dot{\mathbf{Q}}_N ^{-1}$;\\ $\dot{\mathcal{R}}=\hat{\dot{\mathcal{R}}}\times_3\dot{\mathbf{Q}}_3^{-1}\times_4 \cdots \times_N\dot{\mathbf{Q}}_N ^{-1}$;\\ $\dot{\mathcal{V}}=\hat{\dot{\mathcal{V}}}\times_3\dot{\mathbf{Q}}_3^{-1}\times_4\cdots \times_N\dot{\mathbf{Q}}_N ^{-1}$.		
	\end{algorithmic}
\end{algorithm}
\section{Randomized quaternion UTV decompositions}\label{sec:3}
This section is dedicated to the development of randomized QURV and QTURV for the purpose of computing the decomposition in a more efficient manner. This technique is determined by the chosen rank and utilizes a randomized test quaternion matrix with QQR decomposition. The method used for a quaternion matrix is called Compressed Randomized QURV (CoR-QURV), and the approach used for a quaternion tensor is called Compressed Randomized QTURV (CoR-QTURV). 
\subsection{Randomized Algorithms for QURV}
Given a quaternion matrix $\dot{\mathbf{A}}\in\mathbb{H}^{M \times N}$, two integers $l$ and $p$, where $l$ is the target rank and $p$ is the power scheme parameter.  Forming an $N\times l$ quaternion random test matrix as $\dot{\mathbf{\Omega}}=\mathbf{\Omega}_0+\mathbf{\Omega}_1i+\mathbf{\Omega}_2j+\mathbf{\Omega}_3k$, where $\mathbf{\Omega}_i, i=0, 1, 2,3$ are random and independently drawn from the $N(0,1)-$normal distribution. Then, the whole process of CoR-QURV can be summarized in Algorithm \ref{a1.3}.
\begin{algorithm}[htbp]
	\caption{The CoR-QURV decomposition}
	\label{a1.3}
	\begin{algorithmic}[1]
		\REQUIRE   Quaternion matrix data $\dot{\mathbf{A}}\in\mathbb{H}^{M\times N}$,  integers $l$ and the power parameter $p$. 
		\ENSURE    Two unitary quaternion matrices: $\dot{\mathbf{U}}\in\mathbb{H}^{M \times l}$, $\dot{\mathbf{V}}\in\mathbb{H}^{l \times N}$  and an upper triangular quaternion matrix $\dot{\mathbf{R}}\in\mathbb{H}^{l \times l}$ such that $\dot{\mathbf{A}}\approx\dot{\mathbf{U}}\dot{\mathbf{R}}\dot{\mathbf{V}}^H$.
		\STATE Draw an $N\times l$ quaternion random test matrix $\dot{\mathbf{\Omega}}$.
		\STATE  Construct $\dot{\mathbf{Y}}_0=\dot{\mathbf{A}}\dot{\mathbf{\Omega}}$. \STATE \textbf{for} $i=1: p$ \textbf{ do}
		\STATE \qquad Compute $\hat{\dot{\mathbf{Y}}}_i=\dot{\mathbf{A}}^H\dot{\mathbf{Y}}_{i-1}$ and $\dot{\mathbf{Y}}_i=\dot{\mathbf{A}}\hat{\dot{\mathbf{Y}}}_{i}$.
		\STATE \textbf{end for} 
		\STATE Compute the QQR of $\hat{\dot{\mathbf{Y}}}_q$  and $\dot{\mathbf{Y}}_q$ as $\hat{\dot{\mathbf{Y}}}_q=\dot{\mathbf{Q}}_1\dot{\mathbf{R}}_1$ and $\dot{\mathbf{Y}}_q=\dot{\mathbf{Q}}_2\dot{\mathbf{R}}_2$.
		\STATE Compute the QQRCP of  $\dot{\mathbf{M}}=\dot{\mathbf{Q}}_1^H\dot{\mathbf{A}}\dot{\mathbf{Q}_2}$ as $\dot{\mathbf{M}}=\dot{\mathbf{Q}}_3\dot{\mathbf{R}}_3\mathbf{P}^T$.
		\STATE  Form the CoR-QURV decomposition of  $\dot{\mathbf{A}}$, $\dot{\mathbf{U}}=\dot{\mathbf{Q}}_1\dot{\mathbf{Q}}_3$, $\dot{\mathbf{R}}=\dot{\mathbf{R}}_3$, and $\dot{\mathbf{V}}=\dot{\mathbf{Q}}_2\mathbf{P}$.		
	\end{algorithmic}
\end{algorithm}
\begin{remark}
	For the 7th step in Algorithm \ref{a1.3}, when $p=1$, multiplying $\dot{\mathbf{Q}}_2^H\dot{\mathbf{\Omega}}$, then the equation can be rewritten as $\dot{\mathbf{M}}\dot{\mathbf{Q}_2}\dot{\mathbf{\Omega}}=\dot{\mathbf{Q}}_1^H\dot{\mathbf{A}}\dot{\mathbf{Q}_2}\dot{\mathbf{Q}}_2^H\dot{\mathbf{\Omega}}$. Assuming $\dot{\mathbf{A}}\approx\dot{\mathbf{A}}\dot{\mathbf{Q}_2}\dot{\mathbf{Q}}_2^H$, the quaternion matrix $\dot{\mathbf{Q}}_1^H\dot{\mathbf{A}}\dot{\mathbf{Q}}_2$ can be replaced by $\dot{\mathbf{Q}}_1^H\dot{\mathbf{Y}}_0(\dot{\mathbf{Q}}_2^H\dot{\mathbf{\Omega}})^{\dag}$ which requires $\mathcal{O}(Ml^2+l^3+Nl^2)$. However, the computational complexity of computing $(\dot{\mathbf{Q}}_1^H\dot{\mathbf{A}})\dot{\mathbf{Q}}_2$ and $\dot{\mathbf{Q}}_1^H(\dot{\mathbf{A}}\mathbf{Q}_2)$ is $\mathcal{O}(MNl+Nl^2)$ and $\mathcal{O}(MNl+Ml^2)$, respectively. When $l<\min({M, N})$, the cost of obtaining $\dot{\mathbf{Q}}_1^H\dot{\mathbf{Y}}_0\dot{\mathbf{\Omega}}^{\dag}\dot{\mathbf{Q}}_2$ is less than the cost of obtaining $\dot{\mathbf{Q}}_1^H\dot{\mathbf{A}}\dot{\mathbf{Q}}_2$.
\end{remark}

The computational complexity of each step in Algorithm \ref{a1.3} is assessed. In the first step, drawing a quaternion random test matrix $\Omega$ needs $\mathcal{O}(Ml)$. Then, the computational complexity for forming $\hat{\dot{\mathbf{Y}}}_i$ and $\dot{\mathbf{Y}}_i$ both  are $\mathcal{O}(Ml^2)$. The costs of QQR decomposition are $\mathcal{O}(Ml^2)$ and $\mathcal{O}(Nl^2)$, then computing the multiplication of $(\dot{\mathbf{Q}}_1^H\dot{\mathbf{A}})\dot{\mathbf{Q}}_2$ and QQRCP  need $\mathcal{O}(MNl+Nl^2)$ and $\mathcal{O}(l^3)$, respectively. At last, forming $\dot{\mathbf{U}}$ and $\dot{\mathbf{V}}$ require $\mathcal{O}(Ml^2)$ and $\mathcal{O}(Nl)$, respectively. Hence, the cost of CoR-QURV is $\mathcal{O}(Ml+Nl+MNl+3Ml^2+2Nl^2+l^3)$.
\subsection{Randomized Algorithms for QTURV}
The CoR-QTURV decomposition for the $N$th-order quaternion tensor is devised in a manner similar to the proposed QTURV decomposition.
\begin{algorithm}[htbp]
	\caption{The CoR-QTURV decomposition}
	\label{a1.4}
	\begin{algorithmic}[1]
		\REQUIRE   Quaternion tensor data $\dot{\mathcal{A}}\in\mathbb{H}^{{I_1\times I_2\times \cdots \times I_N}}(N\geq3)$, and invertible quaternion matrices $\dot{\mathbf{Q}}_3\in\mathbb{H}^{I_3\times I_3}, \cdots,  \dot{\mathbf{Q}}_N\in\mathbb{H}^{I_N\times I_N}$. Integers $l$ and the power parameter $p$.
		\ENSURE    Two unitary quaternion tensors:
		$\dot{\mathcal{U}}\in\mathbb{H}^{I_1\times l\times \cdots \times I_N}$,
		$\dot{\mathcal{V}}\in\mathbb{H}^{I_2\times l\times \cdots \times I_N}$, and a f-upper quaternion tensor $\dot{\mathcal{R}}\in\mathbb{H}^{l\times l\times \cdots \times I_N}$ such that $\dot{\mathcal{A}}\approx \dot{\mathcal{U}}\star_{QT}\dot{\mathcal{T}}\star_{QT}\dot{\mathcal{V}}^H$.
		\STATE Draw an $I_2\times l$ quaternion random test matrix $\dot{\mathbf{\Omega}}$.
		\STATE Convert $\dot{\mathcal{A}}$ to transform domain: $\hat{\dot{\mathcal{A}}}=\dot{\mathcal{A}}\times_3\dot{\mathbf{Q}}_3\times_4 \cdots \times_N\dot{\mathbf{Q}}_N $.
		\STATE Concatenate frontal slices along the third dimension:\\ $\hat{\dot{\mathcal{A}}}=\text{reshape}(\hat{\dot{\mathcal{A}}},[I_1,  I_2, I_3I_4\cdots I_N])$.
		\STATE \textbf{for} k=1:$I_3I_4\cdots I_N$ \textbf{do}
		\STATE \quad  Construct $\dot{\mathbf{Y}}_0^k=\hat{\dot{\mathcal{A}}}(:, :, k)\dot{\mathbf{\Omega}}$.
		\STATE \quad  \textbf{for} $i=1:p$ \textbf{ do}
		\STATE \qquad Compute $\hat{\dot{\mathbf{Y}}}_i^k=\hat{\dot{\mathcal{A}}}(:, :, k)^H\dot{\mathbf{Y}}_{i-1}^k$ and $\dot{\mathbf{Y}}_i^k=\hat{\dot{\mathcal{A}}}(:, :, k)\hat{\dot{\mathbf{Y}}}_{i}^k$.
		\STATE \quad  \textbf{end for} 
		\STATE	\quad Compute the QQR of $\hat{\dot{\mathbf{Y}}}_q$  and $\dot{\mathbf{Y}}_q$ as $\hat{\dot{\mathbf{Y}}}_q^k=\dot{\mathbf{Q}}_1^k\dot{\mathbf{R}}_1^k$ and $\dot{\mathbf{Y}}_q^k=\dot{\mathbf{Q}}_2^k\dot{\mathbf{R}}_2^k$.
		\STATE \quad Compute the QQRCP of $\dot{\mathbf{M}}^k=(\dot{\mathbf{Q}}_1^{k})^{H}\hat{\dot{\mathcal{A}}}(:, :, k)\mathbf{Q}_2^k$ as  $\dot{\mathbf{M}}^k=\dot{\mathbf{Q}}_3^k\dot{\mathbf{R}}_3^k(\mathbf{P}^{k})^{T}$.
		\STATE \quad Let $\widetilde{\dot{\mathcal{U}}}(:, :, k)=\dot{\mathbf{Q}}_1^k\dot{\mathbf{Q}}_3^k$, $\widetilde{\dot{\mathcal{R}}}(:, :, k)=\dot{\mathbf{R}}_3^k$, $\widetilde{\dot{\mathcal{V}}}(:, :, k)=\dot{\mathbf{Q}}_2^k\mathbf{P}^{k}$
		\STATE \textbf{end for}
		\STATE Reshape into $N$th-order quaternion tensors:\\ 
		$\hat{\dot{\mathcal{U}}}=\text{reshape}(\widetilde{\dot{\mathcal{U}}}(:, :, k), [I_1,  l, I_3,\cdots, I_N])$;\\ $\hat{\dot{\mathcal{R}}}=\text{reshape}(\widetilde{\dot{\mathcal{R}}}(:, :, k), [l,  l, I_3, \cdots, I_N])$;\\ $\hat{\dot{\mathcal{V}}}=\text{reshape}(\widetilde{\dot{\mathcal{V}}}(:, :, k), [I_2,  l, I_3,\cdots, I_N])$.		
		\STATE Return to the original domain:\\ $\dot{\mathcal{U}}=\hat{\dot{\mathcal{U}}}\times_3\dot{\mathbf{Q}}_3^{-1}\times_4\cdots \times_N\dot{\mathbf{Q}}_N ^{-1}$;\\ $\dot{\mathcal{R}}=\hat{\dot{\mathcal{R}}}\times_3\dot{\mathbf{Q}}_3^{-1}\times_4 \cdots \times_N\dot{\mathbf{Q}}_N ^{-1}$;\\ $\dot{\mathcal{V}}=\hat{\dot{\mathcal{V}}}\times_3\dot{\mathbf{Q}}_3^{-1}\times_4\cdots \times_N\dot{\mathbf{Q}}_N ^{-1}$.		
	\end{algorithmic}
\end{algorithm}

The computational complexity of every step in Algorithm \ref{a1.4} is evaluated. Consider a quaternion tensor $\dot{\mathcal{A}}\in\mathbb{H}^{{I_1\times I_2\times I_3}}$ as an example. The transformation employed in the subsequent experiment is QDCT. In the first step, drawing a quaternion random test matrix $\Omega$ needs $\mathcal{O}(I_2l)$, and converting $\dot{\mathcal{A}}$ to transform domain needs $\mathcal{O}(I_1I_2I_3\log(I_3))$. Then, the computational complexity for forming $\hat{\dot{\mathbf{Y}}}_i^k$ and $\dot{\mathbf{Y}}_i^k$ both  is  $\mathcal{O}(I_1I_2I_3l)$. The costs of QQR decomposition is $\mathcal{O}(I_1I_3l^2)$ and $\mathcal{O}(I_2I_3l^2)$, then computing the multiplication of $(\dot{\mathbf{Q}}_1^H\dot{\mathbf{A}})\dot{\mathbf{Q}}_2$ and QQRCP  need $\mathcal{O}(I_1I_2I_3l+I_2I_3l^2)$ and $\mathcal{O}(I_3l^3)$, respectively. Then, forming $\widetilde{\dot{\mathcal{U}}}$ and $\widetilde{\dot{\mathcal{V}}}$ require $\mathcal{O}(I_1I_3l^2)$ and $\mathcal{O}(I_2I_3l)$, respectively. At last, the computational complexity for transforming $\dot{\mathcal{U}}, \dot{\mathcal{R}}$ and $\dot{\mathcal{V}}$ to the original domain is $\mathcal{O}(I_1I_3l\log(I_3))$, $\mathcal{O}(I_2I_3l\log(I_3))$, and $\mathcal{O}(I_3l^2\log(I_3))$, respectively. Hence, the computational complexity of CoR-QURV is $\mathcal{O}(I_2l+I_1I_2I_3\log(I_3)+2I_1I_3l^2+2I_2I_3l^2+I_1I_2I_3l+I_2I_3l+I_3l^3+I_1I_3l\log(I_3)+I_2I_3l\log(I_3)+I_3l^2\log(I_3))$.
\begin{remark}
	Let a quaternion tensor $\dot{\mathcal{A}}\in\mathbb{H}^{{I_1\times I_2\times I_3}}$ as an example. For the $10$th step in Algorithm \ref{a1.4}, when $p=1$, the quaternion matrix $(\dot{\mathbf{Q}}_1^{k})^{H}\hat{\dot{\mathcal{A}}}(:, :, k)\mathbf{Q}_2^k$ can be replaced by $(\dot{\mathbf{Q}}_1^{k})^{H}\dot{\mathbf{Y}}_0^k((\dot{\mathbf{Q}}_2^{k})^{H}\dot{\mathbf{\Omega}})^{\dag}$ which requires $\mathcal{O}(I_1l^2+l^3+ I_2l^2)$ computational complexity for every $k$. However, the computational complexity of computing $((\dot{\mathbf{Q}}_1^{k})^{H}\hat{\dot{\mathcal{A}}}(:, :, k))\mathbf{Q}_2^k$ and $(\dot{\mathbf{Q}}_1^{k})^{H}(\hat{\dot{\mathcal{A}}}(:, :, k)\mathbf{Q}_2^k)$ is $\mathcal{O}(I_1I_2l+I_2l^2)$ and $\mathcal{O}(I_1I_2l+I_1l^2)$, respectively. When $l<\min({ I_1,  I_2})$, the cost of obtaining $(\dot{\mathbf{Q}}_1^{k})^{H}\dot{\mathbf{Y}}_0^k\dot{\mathbf{\Omega}}^{\dag}\dot{\mathbf{Q}}_2^k$ is less than the cost of obtaining $(\dot{\mathbf{Q}}_1^{k})^{H}\hat{\dot{\mathcal{A}}}(:, :, k)\mathbf{Q}_2^k$.
\end{remark}
\section{Theoretical Analysis}\label{TA}
In this section, the theoretical analysis of the above algorithms are giving, including the low-rank approximation of the randomized methods. These approximations are backed by theoretical assurances that ensure their accuracy in terms of the Frobenius norm. To achieve this objective, a theorem initially is provided.
\begin{theorem}
	Let $\dot{\mathbf{Q}}_1\in\mathbb{H}^{M\times l}$, $\dot{\mathbf{Q}}_2\in\mathbb{H}^{N\times l}$ be quaternion matrices constructed by Algorithm \ref{a1.3}. Let $\dot{\mathbf{M}}_K$ be the best rank$-K$ approximation of $\dot{\mathbf{Q}}_1^H\dot{\mathbf{A}}\dot{\mathbf{Q}_2}\in\mathbb{H}^{l\times l}$. Then, $\dot{\mathbf{M}}_K$ is an optimal solution of the following optimization problem:
	\begin{equation}\label{the1.1}
		\min\limits_{\dot{\mathbf{M}}, \text{rank}(\dot{\mathbf{M}})\leq k}\|\dot{\mathbf{A}}-\dot{\mathbf{Q}}_1\dot{\mathbf{M}}\dot{\mathbf{Q}}_2^H\|_F=\|\dot{\mathbf{A}}-\dot{\mathbf{Q}}_1\dot{\mathbf{M}}_K\dot{\mathbf{Q}}_2^H\|_F,
	\end{equation}
	besides,\begin{equation}\label{the1.2} \|\dot{\mathbf{A}}-\dot{\mathbf{Q}}_1\dot{\mathbf{M}}_K\dot{\mathbf{Q}}_2^H\|_F\leq\|\dot{\mathbf{A}}-\dot{\mathbf{A}}_K\|_F+\|\dot{\mathbf{A}}_K-\dot{\mathbf{Q}}_1\dot{\mathbf{Q}}_1^H\dot{\mathbf{A}}_K\|_F+\|\dot{\mathbf{A}}_K-\dot{\mathbf{A}}_K\dot{\mathbf{Q}}_2\dot{\mathbf{Q}}_2^H\|_F.
	\end{equation}
\end{theorem}
\begin{proof}
	Because the signal total energy computed in the spatial domain and the quaternion domain are equal, according to the Parseval theorem in the quaternion domain \cite{DBLP:journals/corr/Hitzer13d, DBLP:journals/cma/BahriHHA08}. The following relations are hold.
	\begin{equation}
		\begin{split}
			\|\dot{\mathbf{A}}-\dot{\mathbf{Q}}_1\dot{\mathbf{M}}\dot{\mathbf{Q}}_2^H\|_F^2&=\|\dot{\mathbf{A}}-\dot{\mathbf{Q}}_1\dot{\mathbf{Q}}_1^H\dot{\mathbf{M}}\dot{\mathbf{Q}}_2\dot{\mathbf{Q}}_2^H+\dot{\mathbf{Q}}_1\dot{\mathbf{Q}}_1^H\dot{\mathbf{M}}\dot{\mathbf{Q}}_2\dot{\mathbf{Q}}_2^H-\dot{\mathbf{Q}}_1\dot{\mathbf{M}}\dot{\mathbf{Q}}_2^H\|_F^2\\&=\|\dot{\mathbf{A}}-\dot{\mathbf{Q}}_1\dot{\mathbf{Q}}_1^H\dot{\mathbf{M}}\dot{\mathbf{Q}}_2\dot{\mathbf{Q}}_2^H\|_F^2+\|\dot{\mathbf{Q}}_1^H\dot{\mathbf{A}}\dot{\mathbf{Q}}_2-\dot{\mathbf{M}}\|_F^2.
		\end{split}
	\end{equation}	
	As pointed out in \cite{eckart1936approximation}, under the Frobenius norm, the truncated QSVD offers the best low-rank approximation of a quaternion matrix in a least-squares sense. Such that, the result in \eqref{the1.1} holds.
	
	Because $\dot{\mathbf{A}}_K$ is the best low-rank approximation to $\dot{\mathcal{A}}$, and $\dot{\mathbf{Q}}_1\dot{\mathbf{M}}_K\dot{\mathbf{Q}}_2$ is the best restricted (within a subspace) low-rank restriction to $\dot{\mathbf{A}}$ with respect to Frobenius norm. This leads to the following result
	\begin{equation}\label{the1.3}
		\begin{split}
			\|\dot{\mathbf{A}}-\dot{\mathbf{A}}_K\|_F\leq\|\dot{\mathbf{A}}-\dot{\mathbf{Q}}_1\dot{\mathbf{M}}_K\dot{\mathbf{Q}}_2\|_F\leq\|\dot{\mathbf{A}}-\dot{\mathbf{Q}}_1\dot{\mathbf{Q}}_1^H\dot{\mathbf{A}}_K\dot{\mathbf{Q}}_2\dot{\mathbf{Q}}_2^H\|_F.
		\end{split}
	\end{equation}	
	The second relation holds because $\dot{\mathbf{Q}}_1\dot{\mathbf{Q}}_1^H\dot{\mathbf{A}}\dot{\mathbf{Q}}_2\dot{\mathbf{Q}}_2^H$ is an undistinguished restricted Frobenius norm approximation to $\dot{\mathbf{A}}$. Next, we prove \eqref{the1.2} holds,
	\begin{equation}\label{the1.4}
		\begin{split}
			&	\|\dot{\mathbf{A}}-\dot{\mathbf{Q}}_1\dot{\mathbf{Q}}_1^H\dot{\mathbf{A}}_K\dot{\mathbf{Q}}_2\dot{\mathbf{Q}}_2^H\|_F^2\\&=\| \dot{\mathbf{A}}-\dot{\mathbf{A}}_K\dot{\mathbf{Q}}_2\dot{\mathbf{Q}}_2^H+\dot{\mathbf{A}}_K\dot{\mathbf{Q}}_2\dot{\mathbf{Q}}_2^H-\dot{\mathbf{Q}}_1\dot{\mathbf{Q}}_1^H\dot{\mathbf{A}}_K\dot{\mathbf{Q}}_2\dot{\mathbf{Q}}_2^H\|_F^2\\&=\| \dot{\mathbf{A}}-\dot{\mathbf{A}}_K\dot{\mathbf{Q}}_2\dot{\mathbf{Q}}_2^H\|_F^2+\|\dot{\mathbf{A}}_K\dot{\mathbf{Q}}_2\dot{\mathbf{Q}}_2^H-\dot{\mathbf{Q}}_1\dot{\mathbf{Q}}_1^H\dot{\mathbf{A}}_K\dot{\mathbf{Q}}_2\dot{\mathbf{Q}}_2^H\|_F^2\\&\quad+\mathfrak{R}[tr((\dot{\mathbf{A}}-\dot{\mathbf{A}}_K\dot{\mathbf{Q}}_2\dot{\mathbf{Q}}_2^H)^H(\dot{\mathbf{A}}_K\dot{\mathbf{Q}}_2\dot{\mathbf{Q}}_2^H-\dot{\mathbf{Q}}_1\dot{\mathbf{Q}}_1^H\dot{\mathbf{A}}_K\dot{\mathbf{Q}}_2\dot{\mathbf{Q}}_2^H))],
		\end{split}
	\end{equation}	
	where the third term in the right side of the equation is equal to
	\begin{equation*}
		\begin{split}
		&\mathfrak{R}[tr((\dot{\mathbf{A}}-\dot{\mathbf{A}}_K\dot{\mathbf{Q}}_2\dot{\mathbf{Q}}_2^H)^H(\mathbf{I}_{M\times M}-\dot{\mathbf{Q}}_1\dot{\mathbf{Q}}_1^H)(\dot{\mathbf{A}}_K\dot{\mathbf{Q}}_2\dot{\mathbf{Q}}_2^H))]\\&=\mathfrak{R}[tr((\mathbf{I}_{M\times M}-\dot{\mathbf{Q}}_1\dot{\mathbf{Q}}_1^H)(\dot{\mathbf{A}}-\dot{\mathbf{A}}_K\dot{\mathbf{Q}}_2\dot{\mathbf{Q}}_2^H)^H(\dot{\mathbf{A}}_K\dot{\mathbf{Q}}_2\dot{\mathbf{Q}}_2^H))]=0.
	\end{split} 
	\end{equation*} 
	
	Combining \eqref{the1.3} and \eqref{the1.4}, then we have
	\begin{equation}
		\begin{split}
			&\|\dot{\mathbf{A}}-\dot{\mathbf{Q}}_1\dot{\mathbf{M}}_K\dot{\mathbf{Q}}_2^H\|_F\\&\leq\| \dot{\mathbf{A}}-\dot{\mathbf{A}}_K\dot{\mathbf{Q}}_2\dot{\mathbf{Q}}_2^H\|_F+\|\dot{\mathbf{A}}_K\dot{\mathbf{Q}}_2\dot{\mathbf{Q}}_2^H-\dot{\mathbf{Q}}_1\dot{\mathbf{Q}}_1^H\dot{\mathbf{A}}_K\dot{\mathbf{Q}}_2\dot{\mathbf{Q}}_2^H\|_F\\&=\| \dot{\mathbf{A}}-\dot{\mathbf{A}}_K+\dot{\mathbf{A}}_K-\dot{\mathbf{A}}_K\dot{\mathbf{Q}}_2\dot{\mathbf{Q}}_2^H\|_F+\|\dot{\mathbf{A}}_K\dot{\mathbf{Q}}_2\dot{\mathbf{Q}}_2^H-\dot{\mathbf{Q}}_1\dot{\mathbf{Q}}_1^H\dot{\mathbf{A}}_K\dot{\mathbf{Q}}_2\dot{\mathbf{Q}}_2^H\|_F\\&\leq\| \dot{\mathbf{A}}-\dot{\mathbf{A}}_K\|_F+\|\dot{\mathbf{A}}_K-\dot{\mathbf{A}}_K\dot{\mathbf{Q}}_2\dot{\mathbf{Q}}_2^H\|_F+\|\dot{\mathbf{A}}_K\dot{\mathbf{Q}}_2\dot{\mathbf{Q}}_2^H-\dot{\mathbf{Q}}_1\dot{\mathbf{Q}}_1^H\dot{\mathbf{A}}_K\dot{\mathbf{Q}}_2\dot{\mathbf{Q}}_2^H\|_F\\&=\| \dot{\mathbf{A}}-\dot{\mathbf{A}}_K\|_F+\|\dot{\mathbf{A}}_K-\dot{\mathbf{A}}_K\dot{\mathbf{Q}}_2\dot{\mathbf{Q}}_2^H\|_F+\|\dot{\mathbf{A}}_K-\dot{\mathbf{Q}}_1\dot{\mathbf{Q}}_1^H\dot{\mathbf{A}}_K\|_F.
		\end{split}
	\end{equation}
\end{proof}

Basing on Algorithm \ref{a1.3}, let  $\hat{\dot{\mathbf{A}}}_{CoR}=\dot{\mathbf{Q}}_1\dot{\mathbf{M}}\dot{\mathbf{Q}}_2^H$ be the CoR-QURV low-rank approximation and $\dot{\mathbf{M}}_K$ be the best rank-K approximation of $\dot{\mathbf{M}}$, then we have $\|\dot{\mathbf{A}}-\hat{\dot{\mathbf{A}}}_{CoR}\|_F\leq\|\dot{\mathbf{A}}-\dot{\mathbf{Q}}_1\dot{\mathbf{M}}_K\dot{\mathbf{Q}}_2^H\|_F$. Then, 
\begin{equation*}
\|\dot{\mathbf{A}}-\hat{\dot{\mathbf{A}}}_{CoR}\|_F\leq\| \dot{\mathbf{A}}-\dot{\mathbf{A}}_K\|_F+\|\dot{\mathbf{A}}_K-\dot{\mathbf{A}}_K\dot{\mathbf{Q}}_2\dot{\mathbf{Q}}_2^H\|_F+\|\dot{\mathbf{A}}_K-\dot{\mathbf{Q}}_1\dot{\mathbf{Q}}_1^H\dot{\mathbf{A}}_K\|_F
\end{equation*}  holds.

Let $\dot{\mathbf{A}}\in\mathbb{H}^{M\times N}$, and the QSVD of $\dot{\mathbf{A}}$ is
\begin{equation}
	\dot{\mathbf{A}}=[\dot{\mathbf{U}}_K,\; \dot{\mathbf{U}}_0]\left[\begin{array}{cc}
		\mathbf{\Sigma}_K& \mathbf{0}  \\
		\mathbf{0} & \mathbf{\Sigma}_0\\
	\end{array}\right][\dot{\mathbf{V}}_K,\; \dot{\mathbf{V}}_0]^H,
\end{equation} 
 where  $\dot{\mathbf{U}}_K\in\mathbb{H}^{M\times K}$,  $\dot{\mathbf{U}}_0\in\mathbb{H}^{M\times M-K}$ and  $\dot{\mathbf{V}}_K\in\mathbb{H}^{N\times K}$,  $\dot{\mathbf{V}}_0\in\mathbb{H}^{N\times N-K}$ are column orthogonal, and $\mathbf{\Sigma}_K\in\mathbb{R}^{K\times K}$.
  All $\mathbf{\Sigma}_i$ are real diagonal matrices, and $\sigma_{k} (k=1,\ldots, K)$ are singular values. Let 
  \begin{equation*}
  \mathbf{\Sigma}_0=[\mathbf{\Sigma}_2,\;\mathbf{\Sigma}_3]\in\mathbb{R}^{M-K\times N-K},
  \end{equation*} 
where $\mathbf{\Sigma}_2\in\mathbb{R}^{l-P-K\times l-P-K}$ and $\mathbf{\Sigma}_2\in\mathbb{R}^{M-l+P\times N-l+P}$.
  $P$ is an integer satisfies $2\leq P+K\leq l$, we define
  \begin{equation*}
  	\tilde{\dot{\Omega}}=\dot{\mathbf{V}}^H\dot{\mathbf{\Omega}}=[\tilde{\dot{\mathbf{\Omega}}}_1^H,\;\tilde{\dot{\mathbf{\Omega}}}_2 ^H]^H\in\mathbb{H}^{N\times l},
  \end{equation*}  where $\tilde{\dot{\Omega}}_1\in\mathbb{H}^{l-P\times l}$ and $\tilde{\dot{\Omega}}_2\in\mathbb{H}^{N-l+P\times l}$. The upper bound of $\|\dot{\mathbf{A}}-\hat{\dot{\mathbf{A}}}_{CoR}\|_F$ is hold in the quaternion domain, and can be summarized in the following theorem.

\begin{theorem}\label{th1.5}
	Assume that the quaternion matrix $\dot{\mathbf{A}}$ has a QSVD, as stated above, and that $2\leq P+K\leq l$. The quaternion matrix $\hat{\dot{\mathbf{A}}}_{CoR}$ is created using Algorithm \ref{a1.3}. Given that $\tilde{\dot{\mathbf{\Omega}}}_1$ has a full row rank. Then 
	\begin{equation}\label{the1.5}
		\|\dot{\mathbf{A}}-\hat{\dot{\mathbf{A}}}_{CoR}\|_F\leq	\|\dot{\mathbf{A}}_0\|_F+\sqrt{\frac{\alpha^2\|\tilde{\dot{\mathbf{\Omega}}}_2\|_2^2\|\tilde{\dot{\mathbf{\Omega}}}_1^\dagger\|_2^2}{1+\beta^2\|\tilde{\dot{\mathbf{\Omega}}}_2\|_2^2\|\tilde{\dot{\mathbf{\Omega}}}_1^\dagger\|_2^2}}+\sqrt{\frac{\eta^2\|\tilde{\dot{\mathbf{\Omega}}}_2\|_2^2\|\tilde{\dot{\mathbf{\Omega}}}_1^\dagger\|_2^2}{1+\tau^2\|\tilde{\dot\mathbf{\mathbf{\Omega}}}}_2\|_2^2\|\tilde{\dot{\mathbf{\Omega}}}_1^\dagger\|_2^2}},
\end{equation}
where $\alpha=\sqrt{K}\frac{\sigma_{l-P+1}^2}{\sigma_{K}}(\frac{\sigma_{l-P+1}}{\sigma_{K}})^{2p}$, $\beta=\frac{\sigma_{l-P+1}^2}{\sigma_{1}\sigma_{K}}(\frac{\sigma_{l-P+1}}{\sigma_{K}})^{2p}$, $\eta=\frac{\sigma_{K}}{\sigma_{l-P+1}}\alpha$, $\tau=\frac{1}{\sigma_{l-P+1}}\beta$.
\end{theorem}
\begin{proof}
Because the singular values of quaternion matrix are real, and the norm function also can be represented by the real counterpart as can be seen in the Appendix, the proof of Theorem \ref{th1.5} is similar to the proof of Theorem 5 in \cite{DBLP:journals/tsp/KalooraziL18}. Likewise, we can construct two quaternion matrices as 
\begin{equation*}
	\dot{\mathbf{W}}=(\mathbf{\Sigma}_3^T\mathbf{\Sigma}_3)^{p+1}\tilde{\dot{\mathbf{\Omega}}}_2\tilde{\dot{\mathbf{\Omega}}}_1^\dagger\left(\begin{array}{c}
		\mathbf{\Sigma}_K^{-2p-2}\\\mathbf{0}_{l-P-K\times K}
	\end{array}\right),\end{equation*}
\begin{equation*}
	\dot{\mathbf{D}}=\mathbf{\Sigma}_3(\mathbf{\Sigma}_3^T\mathbf{\Sigma}_3)^p\tilde{\dot{\mathbf{\Omega}}}_2\tilde{\dot{\mathbf{\Omega}}}_1^\dagger\left(\begin{array}{c}
		\mathbf{\Sigma}_K^{-2p-1}\\\mathbf{0}_{l-P-K\times K}
	\end{array}\right).
\end{equation*}
Besides, 
\begin{equation}\label{e1.1}
	\|\dot{\mathbf{A}}_K-\dot{\mathbf{Q}}_1\dot{\mathbf{Q}}_1^H\dot{\mathbf{A}}_K\|_F\leq\frac{\sqrt{K}\|\dot{\mathbf{D}}\mathbf{\Sigma}_1\|_2\sigma_1}{\sqrt{\|\dot{\mathbf{D}}\mathbf{\Sigma}_1\|_2^2+\sigma_1^2}},
\end{equation}
\begin{equation}\label{e1.2}
	\|\dot{\mathbf{A}}_K-\dot{\mathbf{A}}_K\dot{\mathbf{Q}}_2\dot{\mathbf{Q}}_2^H\|_F\leq\frac{\sqrt{K}\|\dot{\mathbf{W}}\mathbf{\Sigma}_1\|_2\sigma_1}{\sqrt{\|\dot{\mathbf{W}}\mathbf{\Sigma}_1\|_2^2+\sigma_1^2}}\\
\end{equation}
are also hold in the quaternion domain. 

According to the definition of $\dot{\mathbf{W}}$, $\dot{\mathbf{D}}$, and the property of $\|\dot{\mathbf{A}}\|_2=\|\dot{\mathbf{A}}_C\|_2$, where $\dot{\mathbf{A}}_C$ is the column representation of $\dot{\mathbf{A}}$. Then we have 
\begin{equation}\label{e1.3}
	\frac{	\|\dot{\mathbf{W}}\mathbf{\Sigma}_1\|_2}{\|\tilde{\dot{\mathbf{\Omega}}}_2\|_2\|\tilde{\dot{\mathbf{\Omega}}}_1^\dagger\|_2}\leq\frac{\sigma_{l-P+1}^2}{\sigma_{K}}(\frac{\sigma_{l-P+1}}{\sigma_{K}})^{2p}, 
\end{equation}
\begin{equation}\label{e1.4}
	\frac{	\|\dot{\mathbf{D}}\mathbf{\Sigma}_1\|_2}{\|\tilde{\dot{\mathbf{\Omega}}}_2\|_2\|\tilde{\dot{\mathbf{\Omega}}}_1^\dagger\|_2}\leq\sigma_{l-P+1}(\frac{\sigma_{l-P+1}}{\sigma_{K}})^{2p}.
\end{equation}
Because the real function $f(x)=\frac{x}{\sqrt{1+x^2}}$ is monotonically increasing and substituting \eqref{e1.3} and \eqref{e1.4} into \eqref{e1.1} and \eqref{e1.2}, then the theorem holds. 
\end{proof}
\begin{remark}
Theorem \ref{th1.5} shows that the upper bound of $\|\dot{\mathcal{A}}-\dot{\mathcal{A}}_{CoR}\|_F$ mainly depends on the ratio $\frac{\sigma_{l-P+1}}{\sigma_{K}}$. The power approaches reduce the additional components in the error boundary by exponentially decreasing the aforementioned ratios. Therefore, as $p$ increases, the components decrease rapidly, approaching zero in an exponential manner. Nevertheless, this leads to the increase of the computation cost.
\end{remark}
\begin{proposition}\label{p1}
Let $\dot{\mathbf{\Omega}}=\mathbf{\Omega}_0+\mathbf{\Omega}_1i+\mathbf{\Omega}_2j+\mathbf{\Omega}_3k\in\mathbb{H}^{M\times N}, \;M\leq N$, where $\mathbf{\Omega}_i$ $(i=0, 1, 2,3)$ are standard Gaussian matrices. For any $\alpha>0$, we have
\begin{equation*}
\mathbb{E}\left(\sqrt{\frac{\alpha^2\|\dot{\mathbf{\Omega}}\|_2^2}{1+\beta^2\|\dot{\mathbf{\Omega}}\|_2^2}}\right)\leq\sqrt{\frac{\alpha^2\nu^2}{1+\beta^2\nu^2}}, 
\end{equation*} 
where $\nu=3(\sqrt{M}+\sqrt{N})+3$.
\end{proposition}
The proof is given in the Appendix.
\begin{proposition}\label{p2}
Let $\dot{\mathbf{\Omega}}=\mathbf{\Omega}_0+\mathbf{\Omega}_1i+\mathbf{\Omega}_2j+\mathbf{\Omega}_3k\in\mathbb{H}^{M\times N}, \;M\leq N$, where $\mathbf{\Omega}_i$ $(i=0, 1, 2,3)$ are standard Gaussian matrices. For any $\alpha>0$, we have
\begin{equation*}
\mathbb{E}\left(\sqrt{\frac{\alpha^2\|\dot{\mathbf{\Omega}}^\dagger\|_2^2}{1+\beta^2\|\dot{\mathbf{\Omega}}^\dagger\|_2^2}}\right)\leq\sqrt{\frac{\alpha^2\nu^2}{1+\beta^2\nu^2}},
\end{equation*}  
where $\nu=\frac{\mathit{e}\sqrt{4N+2}}{p+1}$.
\end{proposition}
The proof is given in the Appendix.
\begin{theorem}
With the notation of Theorem \ref{th1.5}, for Algorithm \ref{a1.3}, we have 
\begin{equation}\label{theq1.6}
	\mathbb{E}\|\dot{\mathbf{A}}-\hat{\dot{\mathbf{A}}}_{CoR}\|_F\leq	\|\dot{\mathbf{A}}_0\|_F+(1+\gamma_K)\sqrt{K}\nu\sigma_{l-P+1}\gamma_K,
\end{equation}
where $\nu=(3(\sqrt{N}+\sqrt{N})+3)\frac{\mathit{e}\sqrt{4N+2}}{P+1}$ and $\gamma_K=\frac{\sigma_{l-P+1}}{\sigma_{K}}$.
\end{theorem}
\begin{proof}
By utilizing Lemma 2 in \cite{DBLP:journals/siamsc/LiuLJ22}, $\tilde{\dot{\mathbf{\Omega}}}=\dot{\mathbf{V}}^H\dot{\mathbf{\Omega}}=[\tilde{\dot{\mathbf{\Omega}}}_1^H,\;\tilde{\dot{\mathbf{\Omega}}}_2 ^H]^H$ follows $N(0,4I_N)$, where $\tilde{\dot{\mathbf{\Omega}}}_1\in\mathbb{H}^{l-P\times l}$ and $\tilde{\dot{\mathbf{\Omega}}}_2\in\mathbb{H}^{N-l+P\times l}$ are disjoint submatrices. So we first take expectations over $\tilde{\dot{\mathbf{\Omega}}}_2$ and next over $\tilde{\dot{\mathbf{\Omega}}}_1$. By invoking Theorem \ref{th1.5}, Proposition \ref{p1}, and Proposition \ref{p2}, we have
\begin{equation}
	\begin{split}
		&\mathbb{E}\|\dot{\mathbf{A}}-\hat{\dot{\mathbf{A}}}_{CoR}\|_F\\&=\|\dot{\mathbf{A}}_0\|_F+\mathbb{E}_{\tilde{\dot{\mathbf{\Omega}}}_1}\mathbb{E}_{\tilde{\dot{\mathbf{\Omega}}}_2}\left(\sqrt{\frac{\alpha^2\|\tilde{\dot{\mathbf{\Omega}}}_2\|_2^2\|\tilde{\dot{\mathbf{\Omega}}}_1^\dagger\|_2^2}{1+\beta^2\|\tilde{\dot{\mathbf{\Omega}}}_2\|_2^2\|\tilde{\dot{\mathbf{\Omega}}}_1^\dagger\|_2^2}}+\sqrt{\frac{\eta^2\|\tilde{\dot{\mathbf{\Omega}}}_2\|_2^2\|\tilde{\dot{\mathbf{\Omega}}}_1^\dagger\|_2^2}{1+\tau^2\|\tilde{\dot{\mathbf{\Omega}}}_2\|_2^2\|\tilde{\dot{\mathbf{\Omega}}}_1^\dagger\|_2^2}}\right)\\&\leq	\|\dot{\mathbf{A}}_0\|_F+\mathbb{E}_{\tilde{\dot{\mathbf{\Omega}}}_1}\left(\sqrt{\frac{\alpha^2\nu_1^2\|\tilde{\dot{\mathbf{\Omega}}}_1^\dagger\|_2^2}{1+\beta^2\nu_1^2\|\tilde{\dot{\mathbf{\Omega}}}_1^\dagger\|_2^2}}+\sqrt{\frac{\eta^2\nu_1^2\|\tilde{\dot{\mathbf{\Omega}}}_1^\dagger\|_2^2}{1+\tau^2\nu_1^2\|\tilde{\dot{\mathbf{\Omega}}}_1^\dagger\|_2^2}}\right)\\&\leq	\|\dot{\mathbf{A}}_0\|_F+\sqrt{\frac{\alpha^2\nu_1^2\nu_2^2}{1+\beta^2\nu_1^2\nu_2^2}}+\sqrt{\frac{\eta^2\nu_1^2\nu_2^2}{1+\tau^2\nu_1^2\nu_2^2}}\\&\leq	\|\dot{\mathbf{A}}_0\|_F+(\alpha+\eta)\nu_1\nu_2,
	\end{split}
\end{equation}
where $\alpha+\eta=(1+\frac{1}{\gamma_K})\alpha=(1+\frac{1}{\gamma_K})\sqrt{K}\gamma_K\sigma_{l-P+1}=(1+\gamma_K)\sqrt{K}\sigma_{l-P+1}$.
So we have $\mathbb{E}\|\dot{\mathbf{A}}-\hat{\dot{\mathbf{A}}}_{CoR}\|_F\leq\|\dot{\mathbf{A}}_0\|_F+(1+\gamma_K)\sqrt{K}\sigma_{l-P+1}\nu_1\nu_2$, where $\nu_1=3(\sqrt{M}+\sqrt{N})+3$ and $\nu_2 =\frac{\mathit{e}\sqrt{4N+2}}{p+1}$.
\end{proof}

Similarly, we give the theoretical analysis for quaternion tensor cases and take third-order quaternion tensor as an example. The linearly invertible transformation we adopted is QDCT, and the corresponding invertible transformation is Inverse Quaternion Discrete Cosine Transform (IQDCT). The QDCT process and IQDCT process is denoted as $\mathcal{D}(\cdot)$ and $\mathcal{ID}(\cdot)$, respectively. Firstly, basing on Algorithm \ref{a1.4}, supposing that  quaternion tensor data $\dot{\mathcal{A}}\in\mathbb{H}^{{I_1\times I_2\times I_3}}$, let $\hat{\dot{\mathcal{A}}}=\mathcal{D}(\dot{\mathcal{A}})$,  $\dot{\mathbf{Q}}_1^k\in\mathbb{H}^{I_1\times l}$, and $ \dot{\mathbf{Q}}_2^k\in\mathbb{H}^{I_2\times l}$ $(k=1\cdots I_3)$. Let $\hat{\dot{\mathcal{Q}}}_1(:,:,k)=\dot{\mathbf{Q}}_1^k$ and $\hat{\dot{\mathcal{Q}}}_2(:,:,k)=\dot{\mathbf{Q}}_2^k$. Then we have the following theorem.
\begin{theorem}\label{th1.7}
Let $\hat{\dot{\mathcal{Q}}}_1\in\mathbb{H}^{I_1\times l\times I_3}$, $\hat{\dot{\mathcal{Q}}}_2\in\mathbb{H}^{I_2\times l \times I_3}$ be quaternion matrices constructed by Algorithm \ref{a1.4}. Let $\hat{\dot{\mathcal{M}}}_K(:,:,k)$ be the best rank$-K$ approximation of $\hat{\dot{\mathcal{Q}}}_1(:,:,k)^H\hat{\dot{\mathcal{A}}}(:,:,k)\hat{\dot{\mathbf{Q}}}_2(:,:,k)\in\mathbb{H}^{l\times l}$. Let $\dot{\mathcal{Q}}_1=\mathcal{ID}(\hat{\dot{\mathcal{Q}}}_1)$, $\dot{\mathcal{Q}}_2=\mathcal{ID}(\hat{\dot{\mathcal{Q}}}_2)$, and $\dot{\mathcal{M}}_K=\mathcal{ID}(\hat{\dot{\mathcal{M}}}_K)$. Then, $\dot{\mathcal{M}}_K$ is an optimal solution of the following optimization problem:
\begin{equation}\label{the2.1}
	\min\limits_{\dot{\mathcal{M}}, \text{rank}(\dot{\mathcal{M}})\leq k}\|\dot{\mathcal{A}}-\dot{\mathcal{Q}}_1\star_{QT}\dot{\mathcal{M}}\star_{QT}\dot{\mathcal{Q}}_2^H\|_F=\|\dot{\mathcal{A}}-\dot{\mathcal{Q}}_1\star_{QT}\dot{\mathcal{M}}_K\star_{QT}\dot{\mathcal{Q}}_2^H\|_F,
\end{equation}
besides,  the rank of each $\dot{\mathbf{M}}(:,:,k)$ is small than or equal to $K$. Moreover, \begin{equation}\label{the2.2}
	\begin{split} \|\dot{\mathcal{A}}-\dot{\mathcal{Q}}_1\star_{QT}\dot{\mathcal{M}}_K\star_{QT}\dot{\mathcal{Q}}_2^H\|_F\leq\|&\dot{\mathcal{A}}-\dot{\mathcal{A}}_K\|_F+\|\dot{\mathcal{A}}_K-\dot{\mathcal{Q}}_1\star_{QT}\dot{\mathcal{Q}}_1^H\star_{QT}\dot{\mathcal{A}}_K\|_F\\&+\|\dot{\mathcal{A}}_K-\dot{\mathcal{A}}_K\star_{QT}\dot{\mathcal{Q}}_2\star_{QT}\dot{\mathcal{Q}}_2^H\|_F.
	\end{split}
\end{equation}
\end{theorem}
\begin{lemma}
Let $\dot{\mathcal{A}}_{CoR}= \dot{\mathcal{U}}\star_{QT}\dot{\mathcal{T}}\star_{QT}\dot{\mathcal{V}}^H$ as obtained in Algorithm \ref{a1.4}. Then we have $\|\dot{\mathcal{A}}-\dot{\mathcal{A}}_{CoR}\|_F\leq\|\dot{\mathcal{A}}-\dot{\mathcal{Q}}_1\star_{QT}\dot{\mathcal{M}}_K\star_{QT}\dot{\mathcal{Q}}_2^H\|_F$.
\end{lemma}
\begin{proof}
Let $\hat{\dot{\mathcal{A}}}_{CoR}=\mathcal{D}(\dot{\mathcal{A}})$. Basing on Algorithm \ref{a1.4}, the $k$th frontal slice of $\hat{\dot{\mathcal{A}}}_{CoR}(:,:,k)=(\dot{\mathbf{Q}}_1^{k})^{H}\hat{\dot{\mathcal{A}}}(:, :, k)\mathbf{Q}_2^k$. Utilizing the definition $\star_{QT}$ in \cite{DBLP:journals/sigpro/MiaoK23} and the unitary property in QDCT, then we have
\begin{equation*}
\|\dot{\mathcal{A}}-\dot{\mathcal{A}}_{CoR}\|_F^2=\sum\limits_{k=1}\limits^{I_3}\|\hat{\dot{\mathcal{A}}}(:,:,k)-\dot{\mathcal{Q}}_1^k\dot{\mathbf{M}}^k(\dot{\mathcal{Q}}_2^{k})^{H}\|_F^2.
\end{equation*}  Denoting $\dot{\mathbf{M}}^k_K$ as the best rank-$K$ approximation of $\dot{\mathbf{M}}^k$, then
\begin{equation*}
 \|\hat{\dot{\mathcal{A}}}(:,:,k)-\dot{\mathcal{Q}}_1^k\dot{\mathbf{M}}^k(\dot{\mathcal{Q}}_2^{k})^{H}\|_F\leq\|\hat{\dot{\mathcal{A}}}(:,:,k)-\dot{\mathcal{Q}}_1^k\dot{\mathbf{M}}^k_K(\dot{\mathcal{Q}}_2^{k})^{H}\|_F, 
\end{equation*} 
 which completes the proof.
\end{proof}

Let $\dot{\mathcal{A}}\in\mathbb{H}^{I_1\times I_1 \times I_3}$ and $\hat{\dot{\mathcal{A}}}=\mathcal{D}(\dot{\mathcal{A}})$. For every frontal slice of $\hat{\dot{\mathcal{A}}}$, let
\begin{equation*}
\hat{\dot{\mathcal{A}}}(:,:,k)=[\dot{\mathbf{U}}_K^k,\; \dot{\mathbf{U}}_0^k]\left[\begin{array}{cc}
\mathbf{\Sigma}_K^k& \mathbf{0}  \\
\mathbf{0} & \mathbf{\Sigma}_K^k\\
\end{array}\right][\dot{\mathbf{V}}_K^k,\; \dot{\mathbf{V}}_0^k]^H
\end{equation*}
is the QSVD of $\hat{\dot{\mathcal{A}}}(:,:,k)$, where  $\hat{\dot{\mathcal{U}}}_K(:,:,k)=\dot{\mathbf{U}}_K^k\in\mathbb{H}^{I_1\times K}$,  $\hat{\dot{\mathcal{U}}}_0(:,:,k)=\dot{\mathbf{U}}_0^k\in\mathbb{H}^{I_1\times I_1-K}$ and  $\hat{\dot{\mathcal{V}}}_K(:,:,k)=\dot{\mathbf{V}}_K^k\in\mathbb{H}^{I_2\times K}$,  $\hat{\dot{\mathcal{V}}}_0(:,:,k)=\dot{\mathbf{V}}_0^k\in\mathbb{H}^{I_2\times I_2-K}$ are column orthogonal, and $\hat{\dot{\mathcal{S}}}_K(:,:,k)=\mathbf{\Sigma}_K^k\in\mathbb{R}^{K\times K}$, $ \hat{\dot{\mathcal{S}}}_0(:,:,k)=\mathbf{\Sigma}_0^k\in\mathbb{R}^{I_1-K\times I_2-K}$ are real diagonal matrices, and $\sigma_{i}^k$ $(i=1,\ldots, K)$ are singular values. Let $P$ be an integer satisfies $2\leq P+K\leq l$, we define
\begin{equation*}
\tilde{\dot{\Omega}}^k=(\dot{\mathbf{V}}^{k})^{H}\dot{\mathbf{\Omega}}=[(\tilde{\dot{\mathbf{\Omega}}}_1^{k})^{H},\;(\tilde{\dot{\mathbf{\Omega}}}_2 ^{k})^{H}]^H\in\mathbb{H}^{I_2\times l}, 
\end{equation*}
 where $\tilde{\dot{\Omega}}_1^k\in\mathbb{H}^{l-P\times l}$ and $\tilde{\dot{\Omega}}_2^k\in\mathbb{H}^{(I_2-l+P)\times l}$. Basing on the Theorem \ref{th1.5} for quaternion matrix cases, the upper bound of $\|\dot{\mathcal{A}}-\hat{\dot{\mathcal{A}}}_{CoR}\|_F$ is also hold, and can be summarized in the following theorem.

\begin{theorem}\label{th1.8}
Assuming that  the quaternion tensor $\dot{\mathcal{A}}$ has a TQt-SVD as defined in Theorem \ref{th2}, and the the corresponding QSVD in the transformed domain are summarized in the above with $2\leq P+K\leq l$. $\dot{\mathcal{A}}_{CoR}$ is quaternion tensor constructed by Algorithm \ref{a1.4}. Assuming $\tilde{\dot{\mathbf{\Omega}}}_1^k$ is full row rank. Then 
\begin{equation}\label{the1.7}
	\|\dot{\mathcal{A}}-\dot{\mathcal{A}}_{CoR}\|_F\leq	\|\dot{\mathcal{A}}_0\|_F+\sum_{k=1}^{I_3}\sqrt{\frac{\alpha_k^2\|\tilde{\dot{\mathbf{\Omega}}}_2^k\|_2^2\|\tilde{\dot{\mathbf{\Omega}}}_1^{k\dagger}\|_2^2}{1+\beta_k^2\|\tilde{\dot{\mathbf{\Omega}}}_2^k\|_2^2\|\tilde{\dot{\mathbf{\Omega}}}_1^{k\dagger}\|_2^2}}+\sqrt{\frac{\eta_k^2\|\tilde{\dot{\mathbf{\Omega}}}_2^k\|_2^2\|\tilde{\dot{\mathbf{\Omega}}}_1^{k\dagger}\|_2^2}{1+\tau_k^2\|\tilde{\dot\mathbf{\mathbf{\Omega}}}}^k_2\|_2^2\|\tilde{\dot{\mathbf{\Omega}}}_1^{k\dagger}\|_2^2}},
\end{equation}
where $\alpha_k=\sqrt{K}\frac{\sigma_{l-P+1}^{k2}}{\sigma_{K}^k}(\frac{\sigma_{l-P+1}^k}{\sigma_{K}^k})^{2p}$, $\beta_k=\frac{\sigma_{l-P+1}^{k2}}{\sigma_{1}^k\sigma_{K}^k}(\frac{\sigma_{l-P+1}^k}{\sigma_{K}^k})^{2p}$, $\eta=\frac{\sigma_{K}}{\sigma_{l-P+1}^k}\alpha_k$, $\tau_k=\frac{1}{\sigma_{l-P+1}^k}\beta_k$.
\end{theorem}

\begin{theorem}\label{th1.9}
With the notation of Theorem \ref{th1.7}, for Algorithm \ref{a1.4} we have 
\begin{equation}\label{the1.6}
\mathbb{E}\|\dot{\mathcal{A}}-\dot{\mathcal{A}}_{CoR}\|_F\leq	\|\dot{\mathcal{A}}_0\|_F+\frac{\sqrt{K}\nu}{I_3}\sum_{k=1}^{I_3}(1+\gamma_K^k)\sigma_{l-P+1}^k\gamma_K^k,
\end{equation}
where $\nu=(3(\sqrt{N}+\sqrt{N})+3)\frac{\mathit{e}\sqrt{4N+2}}{P+1}$ and $\gamma_K=\frac{\sigma_{l-P+1}^k}{\sigma_{K}^k}$.
\end{theorem}

Basing on the above analysis, proof of the quaternion matrix case, and the proof of Theorem 4.1-4.3 in \cite{DBLP:journals/jscic/CheW22} for the real tensor cases, the proof of Theorem \ref{th1.7}-\ref{th1.9} are similar, so we omit the specific process here.
\section{Experimental Results}
\label{E}
In this section, the efficient of the proposed algorithms are tested. The accuracy and the corresponding time computation are assessed for the approximation of QURV and CoR-QURV by applying them to some simulated low-rank quaternion matrices and real color images. Similarly, the QTURV and CoR-QTURV are tested by applying them to simulated low-rank quaternion tensors and real color videos.  All the experiments were implemented in MATLAB R2019a, on a PC with a 3.00GHz CPU and 8GB RAM. The Quaternion Toolbox for Matlab\footnote{https://qtfm.sourceforge.io} and the Tensor Toolbox for Matlab\footnote{http://www.tensortoolbox.org} are also adopted in the following experiment.
\subsection{Settings}\label{sets}
Each color image is represented by a pure quaternion matrix as $\dot{\mathbf{O}}=\mathbf{O}_R\emph{i}+\mathbf{O}_G\emph{j}+\mathbf{O}_B\emph{k}\in\mathbb{H}^{M \times N}$, where $\mathbf{O}_R$, $\mathbf{O}_G$, and $\mathbf{O}_B$ are the pixel values of RGB channels, separately. A color video is presented as a pure quaternion tenor  $\dot{\mathcal{O}}=\mathcal{O}_R\emph{i}+\mathcal{O}_G\emph{j}+\mathcal{O}_B\emph{k}\in\mathbb{H}^{M \times N \times f}$, where $M \times N$ is the size of each frame of the video, $f$ is the number of frames, and $\mathcal{O}_R$, $\mathcal{O}_G$, and $\mathcal{O}_B$ are the pixel values of RGB channels, separately. For a given quaternion matrix rank $K<\min(M,N)$, the best rank-K approximation is $\dot{\mathbf{A}}_{URV, K}=\dot{\mathbf{U}}_K\dot{\mathbf{R}}_K\dot{\mathbf{V}}_K^H$, where $\dot{\mathbf{U}}_K=\dot{\mathbf{U}}(:, 1:K)$, $\dot{\mathbf{V}}_K=\dot{\mathbf{V}}$, and $\dot{\mathbf{R}}_K=\dot{\mathbf{R}}(1:K, :)$ with $\dot{\mathbf{U}}$, $\dot{\mathbf{V}}$, and $\dot{\mathbf{R}}$ are derived from Algorithm \ref{a1.0} or Algorithm \ref{a1.3}. For a given quaternion tensor, the best TQt-rank-K approximation is $\dot{\mathcal{A}}_{URV, K}=\dot{\mathcal{U}}_K\star_{QT}\dot{\mathcal{R}}_K\star_{QT}\dot{\mathcal{V}}_K^H$, where $\dot{\mathcal{U}}_K=\dot{\mathcal{U}}(:, 1:K, :)$, $\dot{\mathcal{V}}_K=\dot{\mathcal{V}}$, and $\dot{\mathcal{R}}_K=\dot{\mathcal{R}}(1:K, :, :)$ with $\dot{\mathcal{U}}$, $\dot{\mathcal{V}}$, and $\dot{\mathcal{R}}$ are derived from Algorithm \ref{a1.2} or Algorithm \ref{a1.4}. 

The running time is tested by the pair ``tic-toc"(in seconds). The Relative Error (RE) are defined as $\frac{\|\dot{\mathbf{A}}-\dot{\mathbf{A}}_{Rec}\|_F}{\|\dot{\mathbf{A}}\|_F}$, where $\dot{\mathbf{A}}_{Rec}$ is the recovered result of $\dot{\mathbf{A}}$ and $\frac{\|\dot{\mathcal{A}}-\dot{\mathcal{A}}_{Rec}\|_F}{\|\dot{\mathcal{A}}\|_F}$, where $\dot{\mathcal{A}}_{Rec}$ is the recovered result of $\dot{\mathcal{A}}$.
\subsection{Testing the proposed QURV and CoR-QURV algorithms}
In this section,  synthetic quaternion matrices and several color images are used to test the efficiency of the proposed QURV and CoR-QURV. The truncated-QSVD is approximated by $\dot{\mathbf{U}}(:, 1:K)\dot{\mathbf{R}}(1:K, 1:K)\dot{\mathbf{V}}(:,1:K)^H$, where $\dot{\mathbf{U}}$,  $\dot{\mathbf{S}}$ and, $\dot{\mathbf{V}}$ are obtained by QSVD. The truncated-QQRCP is approximated by $\dot{\mathbf{Q}}(:,1:K)\breve{\dot{\mathbf{R}}}(1:K,:)$, where $\dot{\mathbf{Q}}$, $\dot{\mathbf{R}}$, and $\mathbf{P}$ are obtained by QQRCP, and $\breve{\dot{\mathbf{R}}}=\dot{\mathbf{R}}\mathbf{P}^T$. 
\begin{example}\label{eg1}
	Consider a quaternion matrix $\dot{\mathbf{A}}$ is specified in the following format
	\begin{equation}
		\dot{\mathbf{A}}=\dot{\mathbf{P}}\dot{\mathbf{Q}}^H,
	\end{equation}
	where $\dot{\mathbf{P}}\in\mathbb{H}^{500 \times 100}$ and $\dot{\mathbf{Q}}\in\mathbb{H}^{500 \times 100}$ are two random quaternion matrices. 
\end{example}
Let the truncated number $K=10:10:100$, the comparison of  these three methods are displayed in Figure \ref{f1}.
It is observed from Figure \ref{f1} that regarding the RE, truncated QSVD outperforms truncated QURV and truncated QQRCP. Truncated QQRCP performs the worst in terms of RE. In terms of the running time, truncated QQRCP is the most time-saving method, followed by the truncated QURV method, while the truncated QSVD is the least efficient.
\begin{figure}[htbp]
	\centering
	\subfigure{
		\begin{minipage}{6cm}
			\centering
			\includegraphics[width = 6cm]{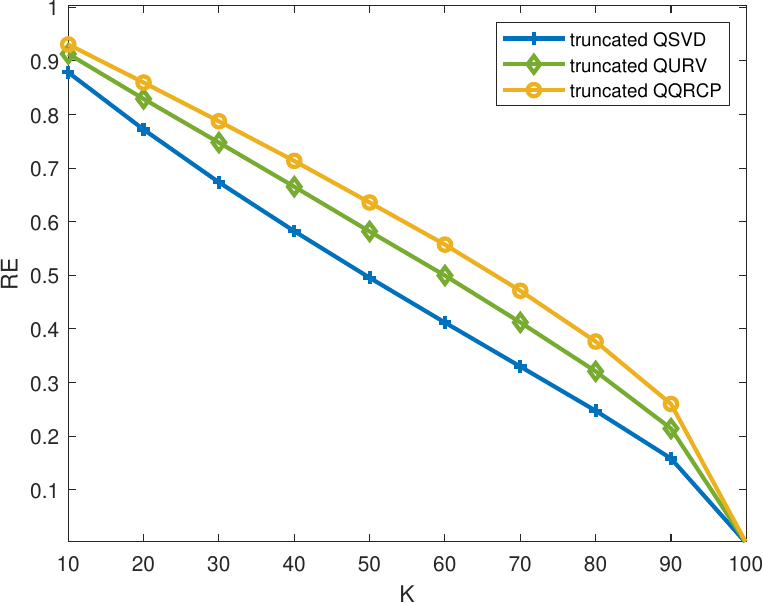}
		\end{minipage}
	}
	\subfigure{
		\begin{minipage}{6cm}
			\centering
			\includegraphics[width = 6cm]{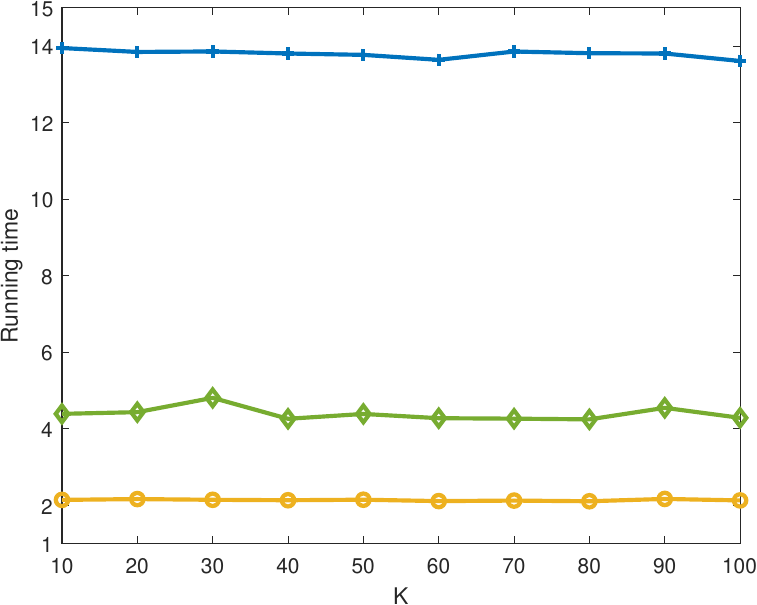}
		\end{minipage}
	}
	\caption{Comparison of  the RE and running time of implementing truncated-QURV truncated-QSVD, and truncated-QQRPCP to  Example \ref{eg1} with $K=10, 20, \cdots, 100.$}	\label{f1}
\end{figure}
\begin{example}\label{eg2}
	Consider a quaternion matrix $\dot{\mathbf{A}}$ is specified in the following format
	\begin{equation}
		\dot{\mathbf{A}}=\dot{\mathbf{U}}\mathbf{\Sigma}\dot{\mathbf{V}}^H,
	\end{equation}
	where $\dot{\mathbf{U}}\in\mathbb{H}^{500 \times 500}$ and $\dot{\mathbf{V}}\in\mathbb{H}^{500 \times 500}$ are two unitary  quaternion matrices derived from computing QSVD of a random quaternion matrix $\bar{\dot{\mathbf{A}}}\in\mathbb{H}^{500 \times 500}.$ $\mathbf{\Sigma}$ is a real diagonal matrix with the $i$th diagonal element is $1/i^2$ $( \mathbf{\Sigma}=diag(1,1/2^2, 1/3^2,\cdots, 1/500^2))$.
\end{example}
Let $K=2:2:20$, the comparison of  these three methods are displayed in Figure \ref{f2}.
\begin{figure}[htbp]	
	\centering
	\subfigure{
		\begin{minipage}{6cm}
			\centering
			\includegraphics[width = 6cm]{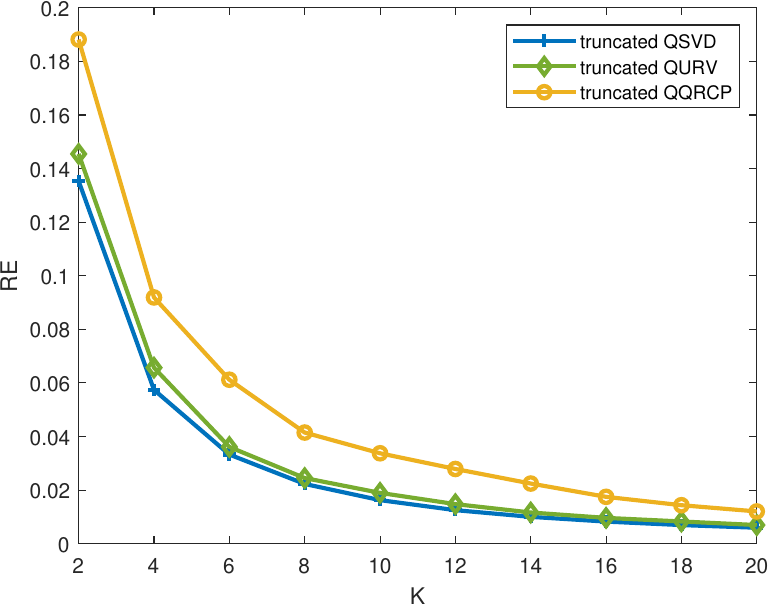}
		\end{minipage}
	}
	\subfigure{
		\begin{minipage}{6cm}
			\centering
			\includegraphics[width = 6cm]{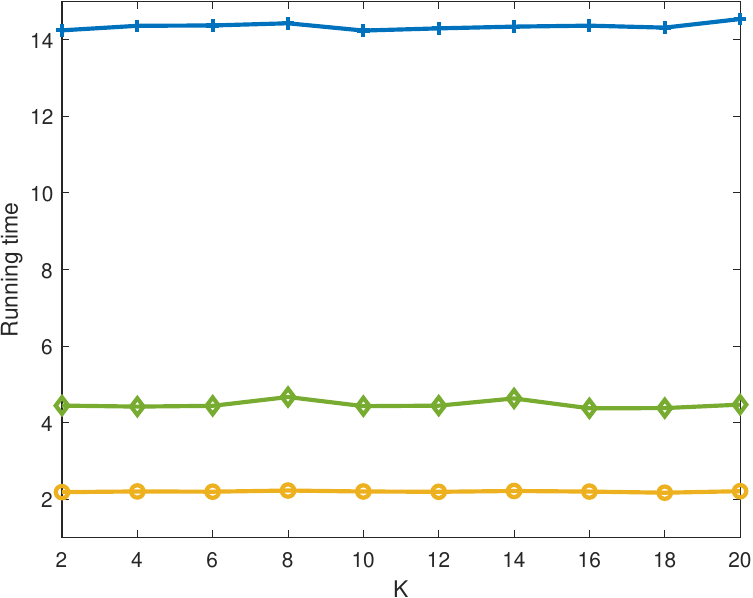}
		\end{minipage}
	}
	\caption{Comparison of  the RE and running time of implementing truncated-QURV truncated-QSVD, and truncated-QQRPCP to  Example \ref{eg2} with $K=2, 4, \cdots, 20.$}\label{f2}
\end{figure}
It is observed from Figure \ref{f2} that truncated QQRCP exhibits the worse performance in terms of RE, especially when the truncated number $K$ is small. QSVD and QURV are comparable regarding the RE. In terms of the running time, truncated QQRCP saves the greatest time, followed by truncated QURV and truncated QSVD.

Then the comparison of randomized strategy is given by utilizing CoR-QURV and randQSVD to color images with the power parameter $p=0, 1, 2$.
\begin{example}\label{eg3}
	Consider two color image as test images (can be seen in Figure \ref{f3}). The size of ``Flower" is $500\times500$ and  the size of ``House" is $1024\times764$. The quaternion matrix representation for color image is given in subsection \ref{sets}. 
	\begin{figure}[htpb]	
		\centering
		\subfigure{
			\begin{minipage}{4cm}
				\centering
				\includegraphics[scale=0.22]{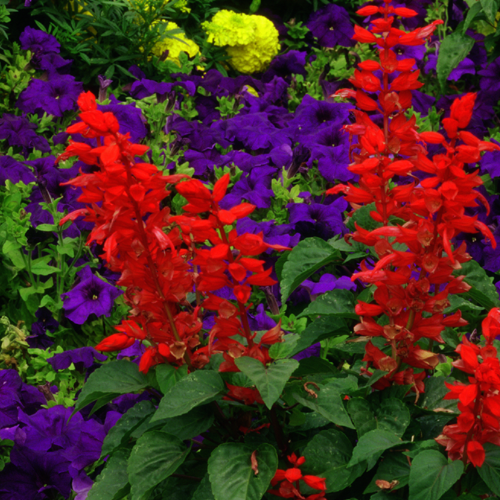}
			\end{minipage}
		}
		\subfigure{
			\begin{minipage}{6cm}
				\centering
				\includegraphics[scale=0.15]{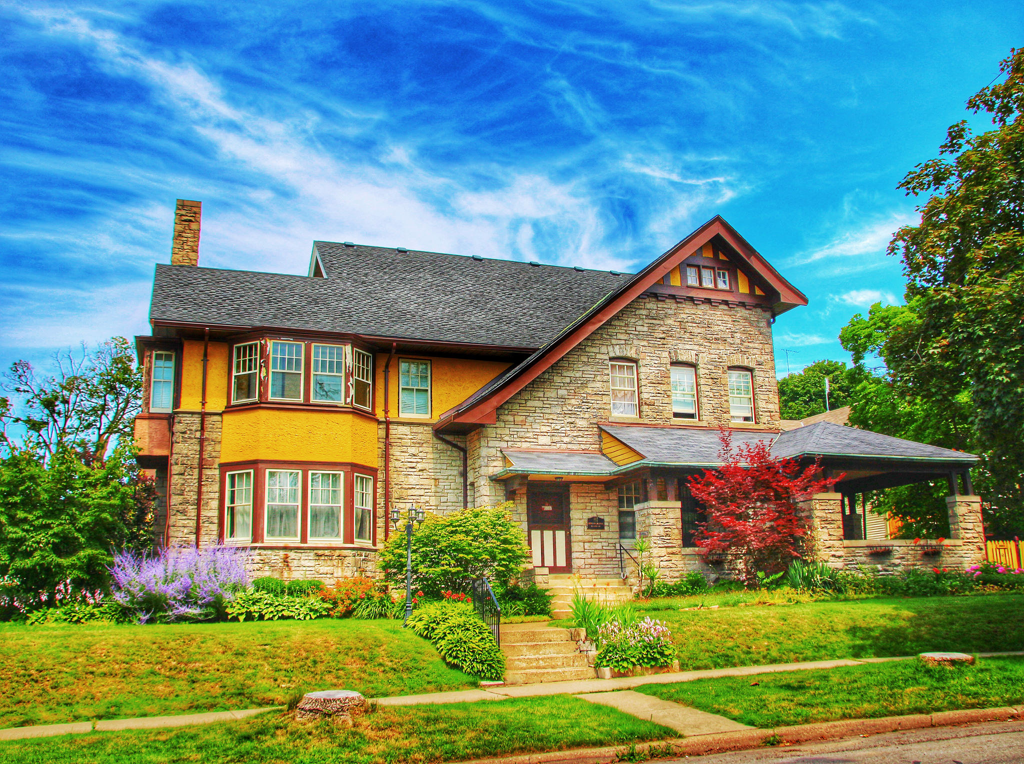}
			\end{minipage}
		}
		\caption{The test images: ``Flower" and ``House".}\label{f3}
	\end{figure}
\end{example}
For image ``Flower", let $K=20:20:200$, the comparison of  CoR-QURV and randQSVD are displayed in Figure \ref{f4}.
\begin{figure}[htbp]	
	\centering
	\subfigure{
		\begin{minipage}{6cm}
			\centering
			\includegraphics[width = 6cm]{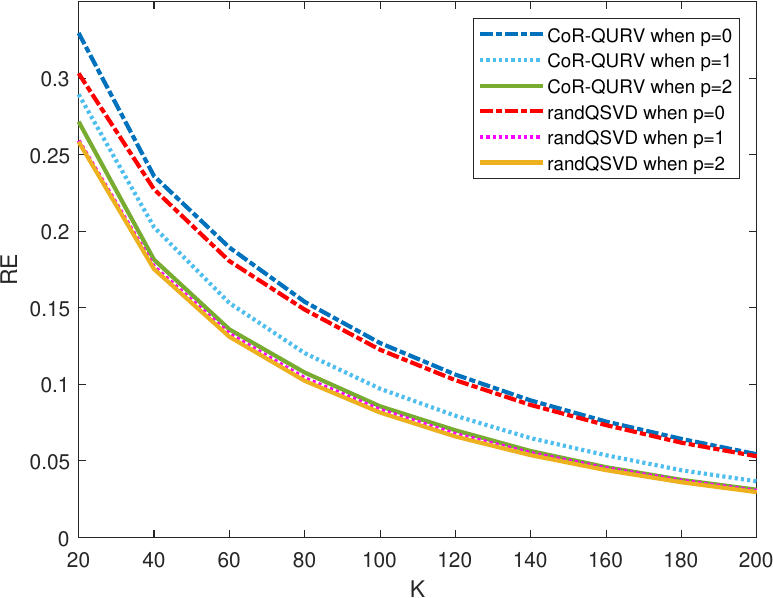}
		\end{minipage}
	}
	\subfigure{
		\begin{minipage}{6cm}
			\centering
			\includegraphics[width = 6cm]{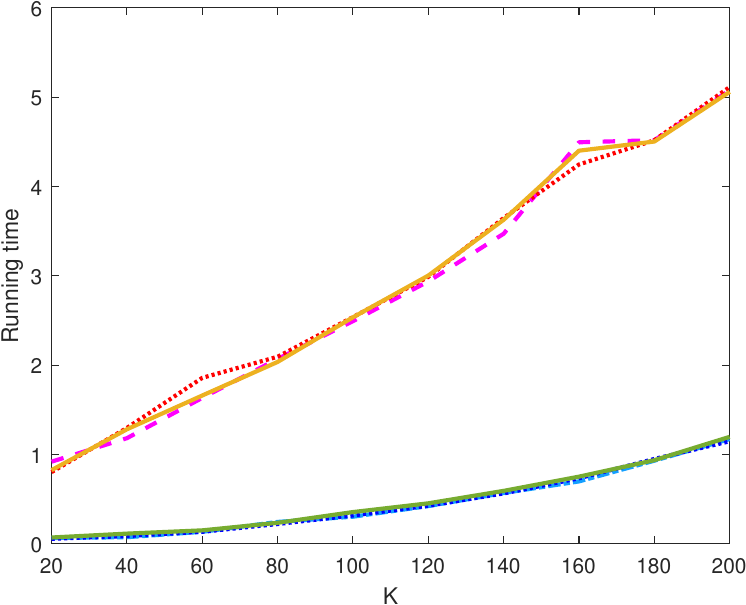}
		\end{minipage}
	}
	\caption{Comparison of  the RE and running time of implementing CoR-QURV and randQSVD to  Example \ref{eg3} with $K=20, 40, \cdots, 200.$}\label{f4}
\end{figure}

For image ``House", let $K=40:40:400$, the comparison of  CoR-QURV and randQSVD are displayed in Figure \ref{f5}.
\begin{figure}[htbp]	
	\centering
	\subfigure{
		\begin{minipage}{6cm}
			\centering
			\includegraphics[width = 6cm]{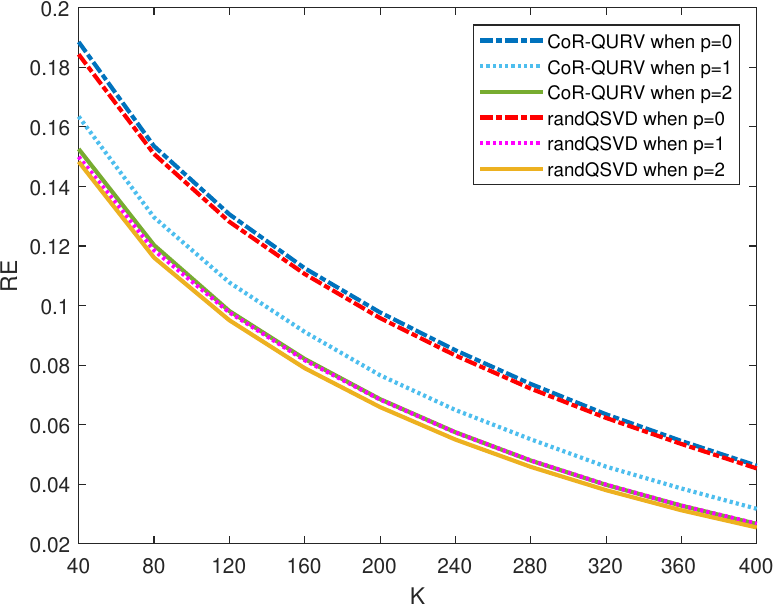}
		\end{minipage}
	}
	\subfigure{
		\begin{minipage}{6.cm}
			\centering
			\includegraphics[width = 6.cm]{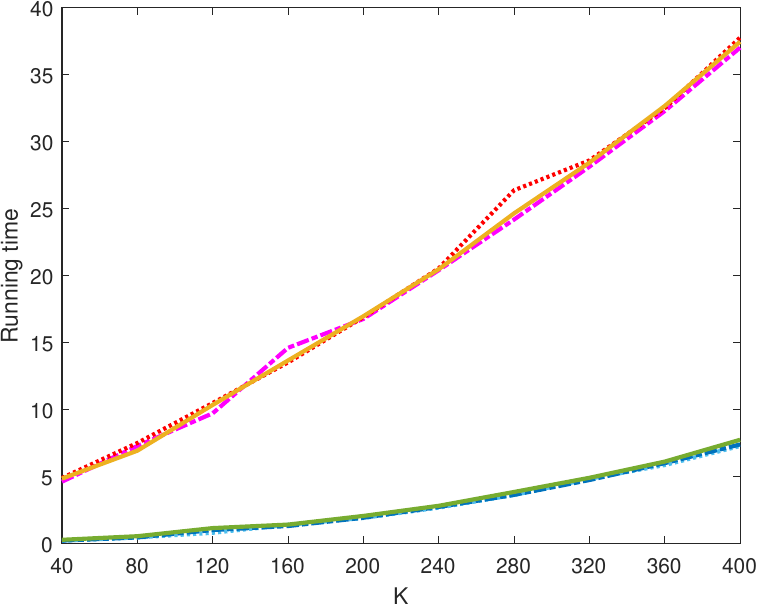}
		\end{minipage}
	}
	\caption{Comparison of  the RE and running time of implementing CoR-QURV and randQSVD to  Example \ref{eg3} with $K=40, 80, \cdots, 400.$}\label{f5}
\end{figure}

Observing from Figure \ref{f4}-\ref{f5}, as the truncated number increases in size, the RE decreases and becomes more stable. As the truncated number increases in size, the techniques require more time to compute. Besides, Figure \ref{f4}-\ref{f5} also demonstrate that the CoR-QURV exhibits slightly lower performance in terms of RE compared to randQSVD when they utilize the same power parameter $p$. When $p=0$, the RE are more unfavorable compared to when $p=1, 2$. When the value of $p$ is equal to 2, the RE can achieve the optimal outcome for randQSVD and CoR-QURV. CoR-QURV outperforms randQSVD in terms of running time for all values of $p$. 

Next, the visual results of image ``House" are shown in Figure \ref{f101}. In this experiment, the power parameter $p=1$, and the rank-K ($K=20, 40, 400$) approximation of randQSVD is given in the first column. Meanwhile, the second column is the recovery of  CoR-QURV.
\begin{figure}[htbp]	
	\centering
	\subfigure[randQSVD $K=20$]{
		\begin{minipage}{6.cm}
			\centering
			\includegraphics[width = 6.cm]{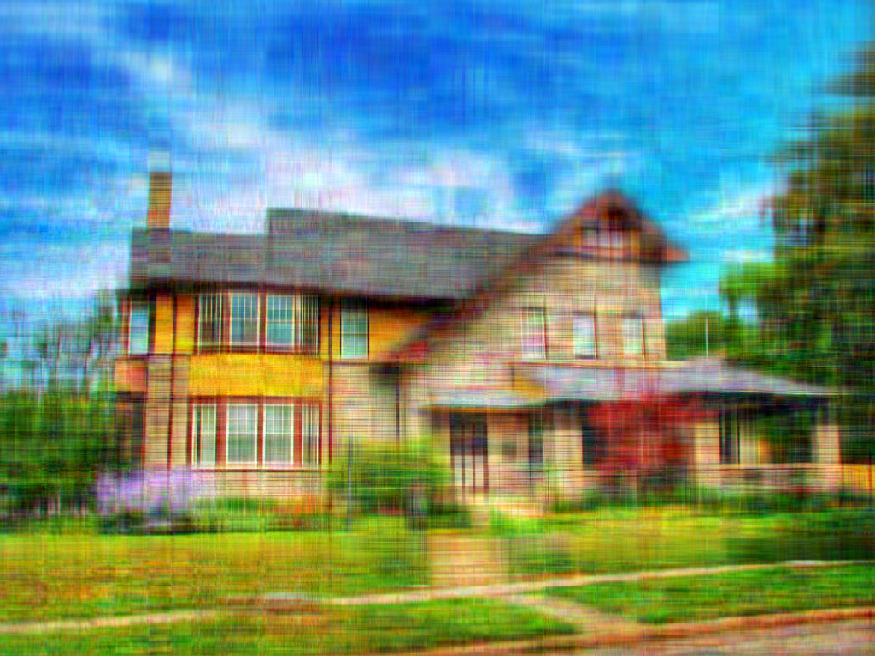}
			
		\end{minipage}
	}
	\subfigure[CoR-QURV $K=20$]{
		\begin{minipage}{6.cm}
			\centering
			\includegraphics[width = 6.cm]{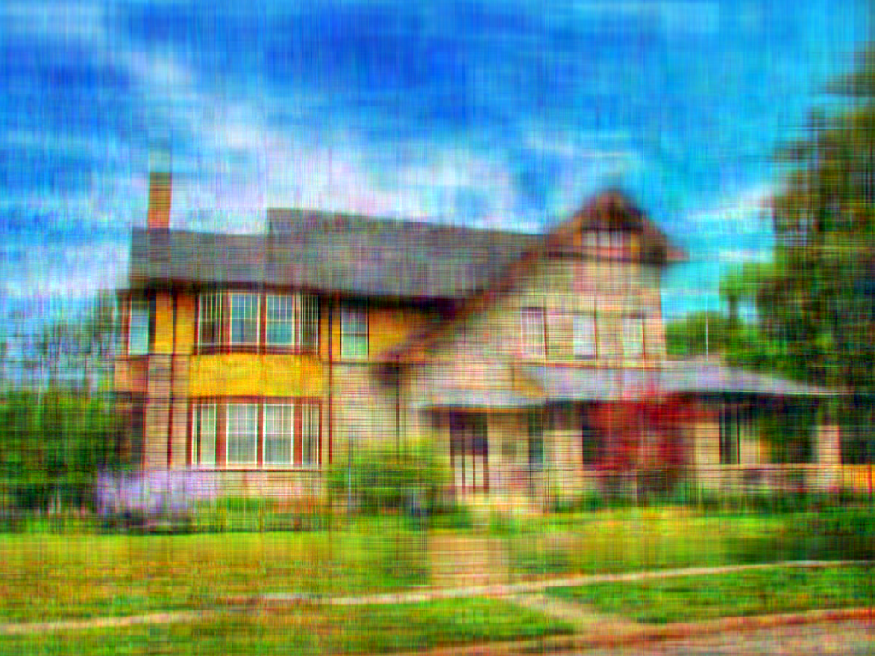}
		\end{minipage}
	}\\
	\subfigure[randQSVD $K=40$]{
		\begin{minipage}{6.cm}
			\centering
			\includegraphics[width = 6.cm]{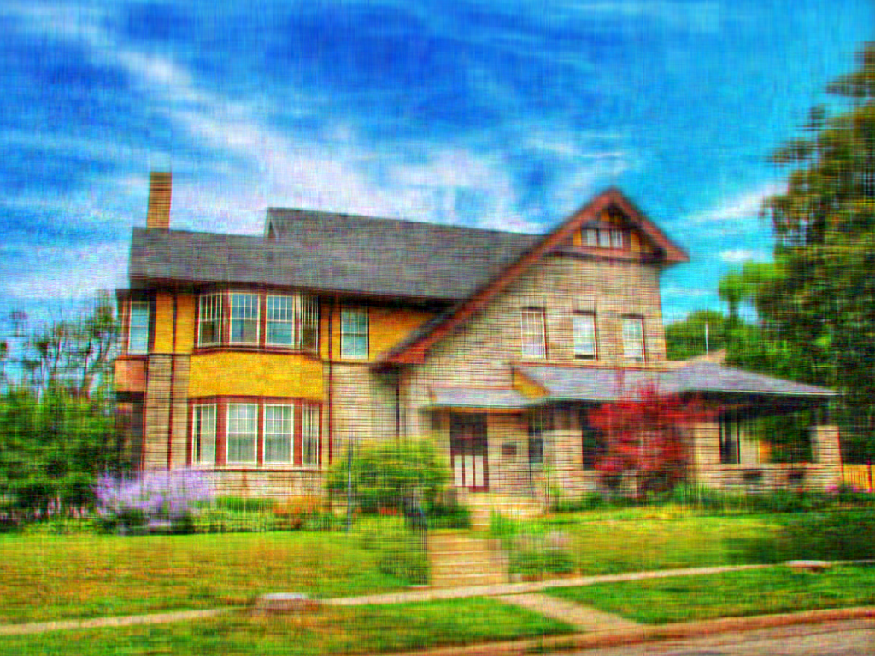}
			
		\end{minipage}
	}
	\subfigure[CoR-QURV $K=40$]{
		\begin{minipage}{6.cm}
			\centering
			\includegraphics[width = 6.cm]{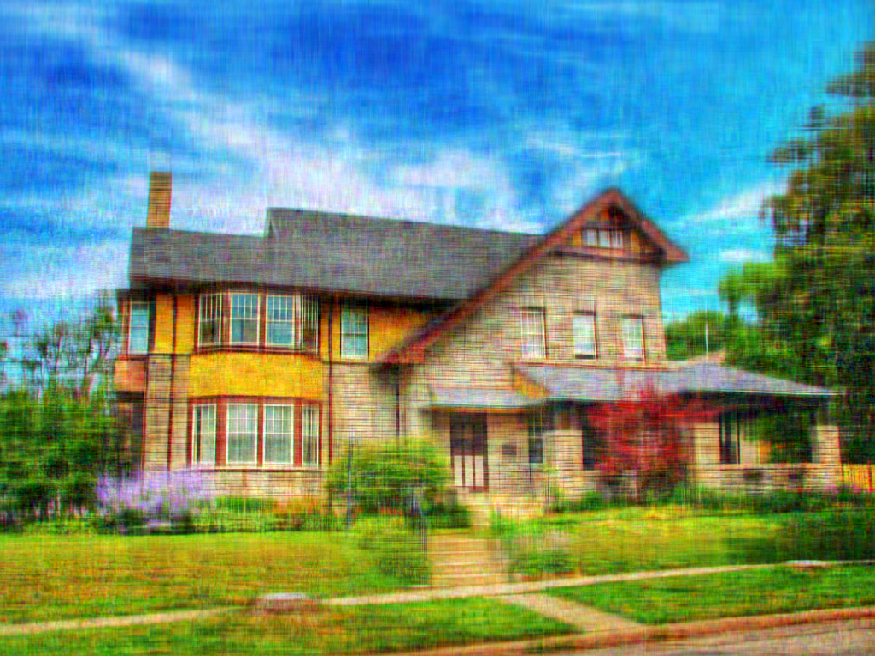}
		\end{minipage}
	}\\
	\subfigure[randQSVD $K=400$]{
		\begin{minipage}{6.cm}
			\centering
			\includegraphics[width = 6.cm]{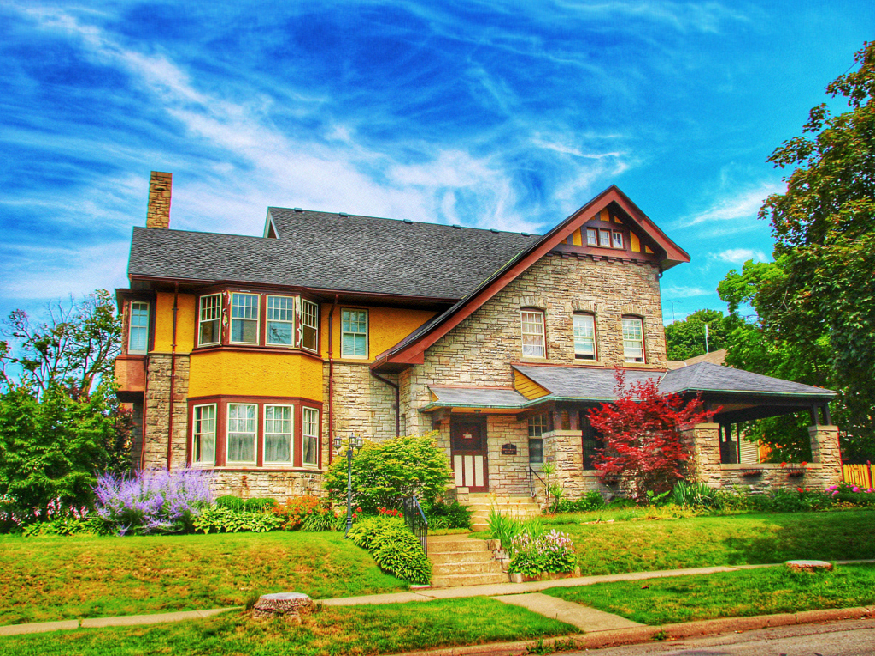}
			
		\end{minipage}
	}
	\subfigure[CoR-QURV $K=400$]{
		\begin{minipage}{6.cm}
			\centering
			\includegraphics[width = 6.cm]{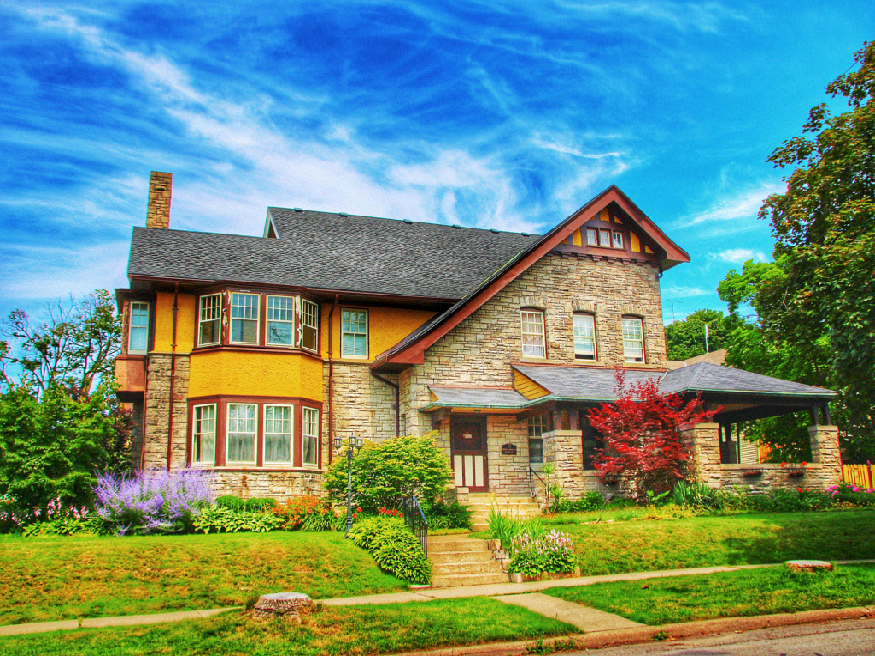}
		\end{minipage}
	}\\
	
	\caption{The rank-K approximation of  ``House" by utilizing CoR-QTURV and randQTSVD. }\label{f101}
\end{figure}
Observing from visual results in Figure \ref{f101}, there is minimal distinction between visual impact of the original image and the recovered image when $K=400$. When the truncated number is tiny, both CoR-QTURV and randQTSVD algorithms can approximately restore the original image, and there is not much visual difference between them.

\subsection{Testing the proposed QTURV and CoR-QTURV algorithms}
In this section,  synthetic quaternion tensors and several color videos are used to test the efficiency of the proposed QTURV and CoR-QTURV. The  TQt-rank-K approximation of truncated-QTSVD is approximated by $\dot{\mathcal{U}}(:,1:K, :)\star_{QT}\dot{\mathcal{R}}_K(1:K,1:K, :)\star_{QT}\dot{\mathcal{V}}_K(:,1:K, :)^H$, where $\dot{\mathcal{U}}$,  $\dot{\mathcal{S}}$ and, $\dot{\mathcal{V}}$ are obtained by QTSVD. 
\begin{example}\label{eg4}
	Consider a quaternion matrix $\dot{\mathcal{A}}$ is specified in the following format
	\begin{equation}
		\dot{\mathcal{A}}=\dot{\mathcal{P}}\star_{QT}\dot{\mathcal{Q}}^H,
	\end{equation}
	where $\dot{\mathcal{P}}\in\mathbb{H}^{300 \times 20 \times 100}$ and $\dot{\mathcal{Q}}\in\mathbb{H}^{300 \times 20 \times 100}$ are two random quaternion tensors. 
\end{example}
The TQt-rank is set to $K=1:20$. The comparison between QTURV and QTSVD is shown in Figure \ref{f6}. Figure \ref{f6} demonstrates that truncated QTSVD performs better than truncated QTURV in terms of the RE. Truncated QTURV is more time-efficient than truncated QTSVD.
\begin{figure}[h]
	\centering
	\subfigure{
		\begin{minipage}{6.cm}
			\centering
			\includegraphics[width = 6.cm]{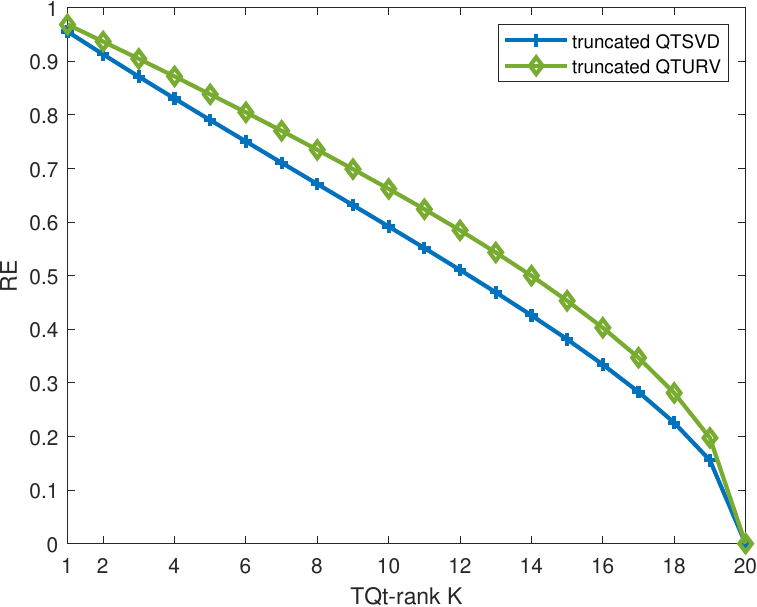}
		\end{minipage}
	}
	\subfigure{
		\begin{minipage}{6.cm}
			\centering
			\includegraphics[width = 6.cm]{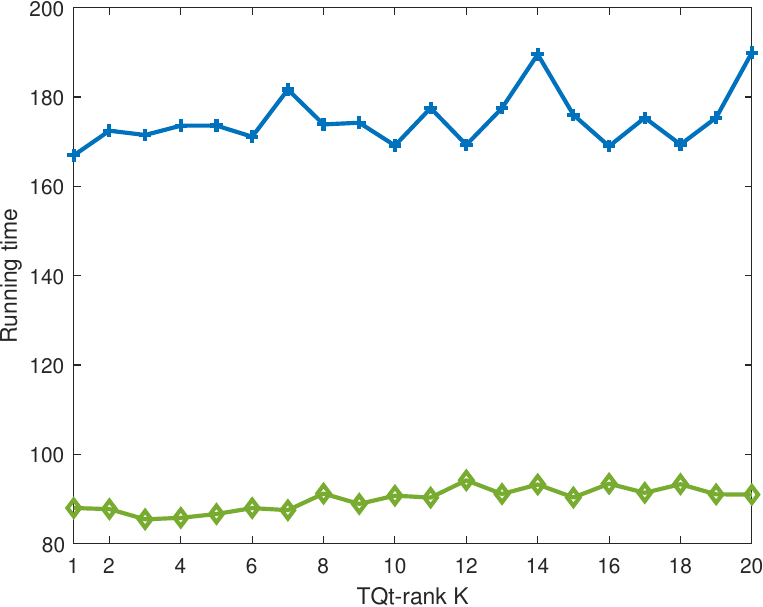}
		\end{minipage}
	}
	\caption{Comparison of  the RE and running time of implementing truncated-QTSVD   and truncated-QTURV to  Example \ref{eg4} with TQt-rank $K=1, 2, \cdots, 20.$}	\label{f6}
\end{figure}

\begin{example}\label{eg5}
	Consider a quaternion tensor  $\dot{\mathcal{A}}$ is specified in the following format
	\begin{equation}
		\dot{\mathcal{A}}=\dot{\mathcal{U}}\star_{QT}\mathcal{D}\star_{QT}\dot{\mathcal{V}}^H,
	\end{equation}
	where $\dot{\mathcal{U}}\in\mathbb{H}^{300 \times 300 \times30}$ and $\dot{\mathcal{V}}\in\mathbb{H}^{300 \times 300 \times 30}$ are two unitary  quaternion tensors derived from computing QTSVD of a random quaternion matrix $\bar{\dot{\mathbf{A}}}\in\mathbb{H}^{300 \times 300 \times 30}.$ $\mathcal{D}$ is a real diagonal tensor that is stacked by diagonal matrices with the $i$th diagonal element is $1/i^2$ $(\mathcal{D}(:, :, k)=diag(1,1/2^2, 1/3^2,\cdots, 1/300^2))$, $k=1:30$.
\end{example}
Let the TQt-rank $K=1: 20$, the comparison of  QTURV with QTSVD are displayed in Figure \ref{f7}.  Figure \ref{f7} illustrates that the truncated QTSVD method outperforms the truncated QTURV method in terms of the RE. As the TQt-rank increases, the discrepancy in RE between truncated QTSVD and truncated QTURV is diminishing. In addition, the truncated QTURV method is more time-efficient than the truncated QTSVD method in all situations.  
\begin{figure}[h]
	\centering
	\subfigure{
		\begin{minipage}{6cm}
			\centering
			\includegraphics[width = 6.cm]{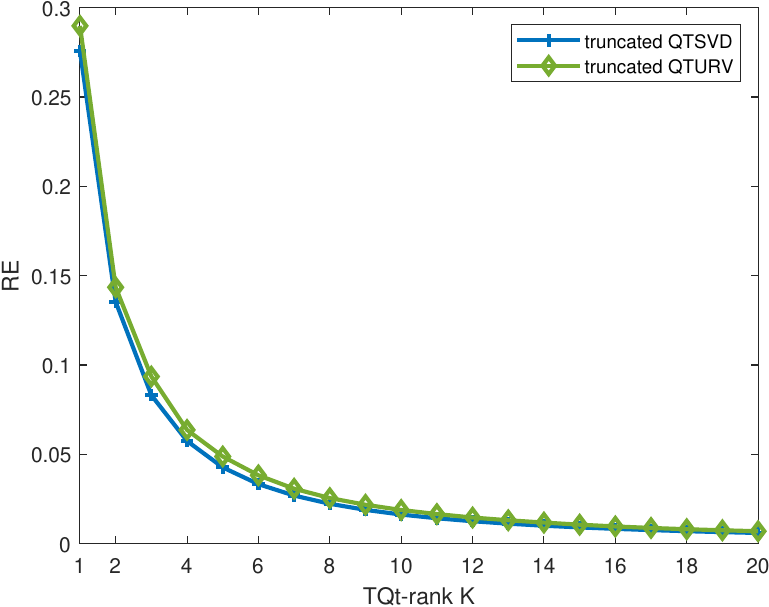}
		\end{minipage}
	}
	\subfigure{
		\begin{minipage}{6.cm}
			\centering
			\includegraphics[width = 6.cm]{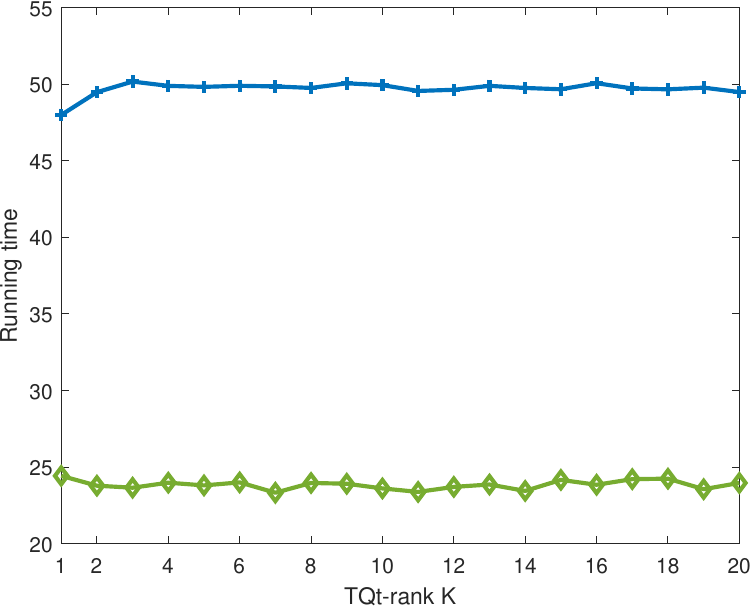}
		\end{minipage}
	}
	\caption{Comparison of  the RE and running time of implementing truncated-QTSVD   and truncated-QTURV to  Example \ref{eg5} with TQt-rank $K=1, 2, \cdots, 20.$}	\label{f7}
\end{figure}

Then the comparison of randomized strategy is given by utilizing CoR-QTURV and randQTSVD to color videos with the power parameter $p=0, 1, 2$.
\begin{example}\label{eg6}
	Consider two color videos as test data (can be seen in Figure \ref{f8}). The size of  ``Football" is $288\times352\times125$ and  the size of  ``Landscape" is $288\times352\times250$. The quaternion tensor representation for color video is given in subsection \ref{sets}. 
	\begin{figure}[htpb]	
		\centering
		\subfigure{
			\begin{minipage}{6.cm}
				\centering
				\includegraphics[scale=0.25]{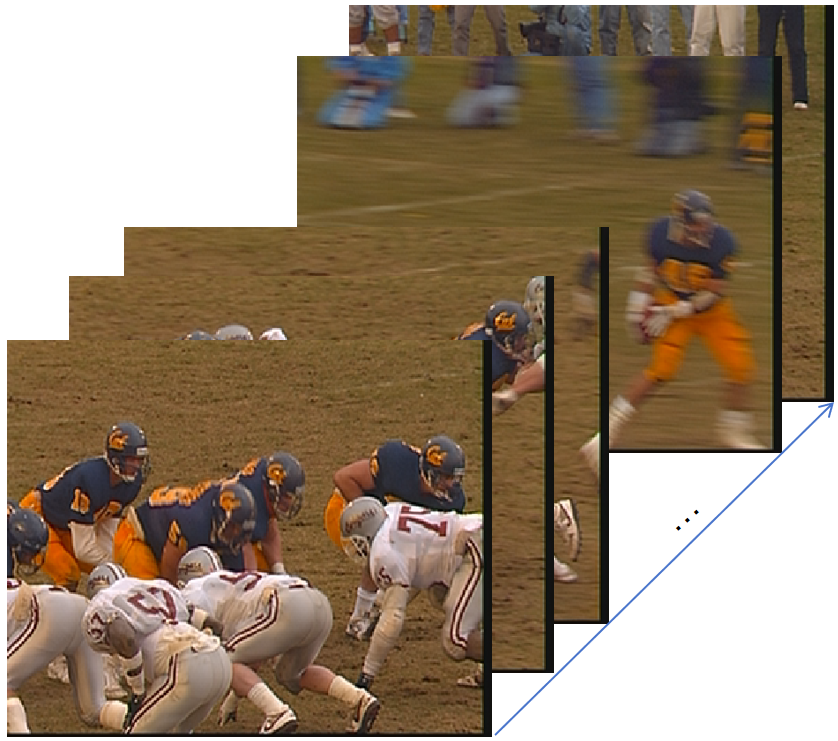}
			\end{minipage}
		}
		\subfigure{
			\begin{minipage}{6.cm}
				\centering
				\includegraphics[scale=0.25]{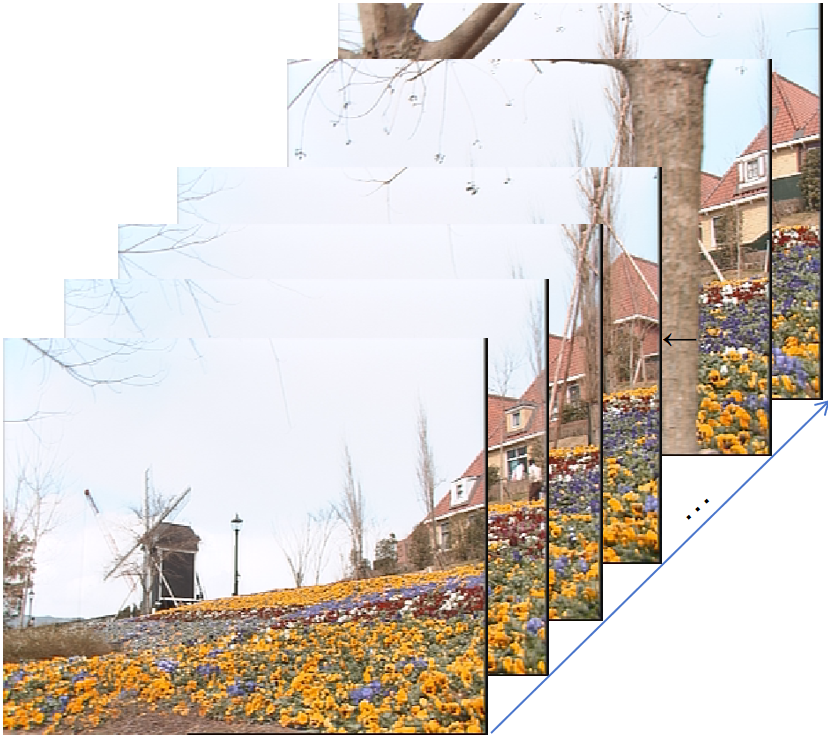}
			\end{minipage}
		}
		\caption{The test videos: ``Football" and ``Landscape".}\label{f8}
	\end{figure}
\end{example}
For video ``Football", let the TQt-rank $K=6:6:60$, the comparison of  CoR-QTURV and randQTSVD is displayed in Figure \ref{f9}.
\begin{figure}[htpb]	
	\centering
	\subfigure{
		\begin{minipage}{6.cm}
			\centering
			\includegraphics[width = 6.cm]{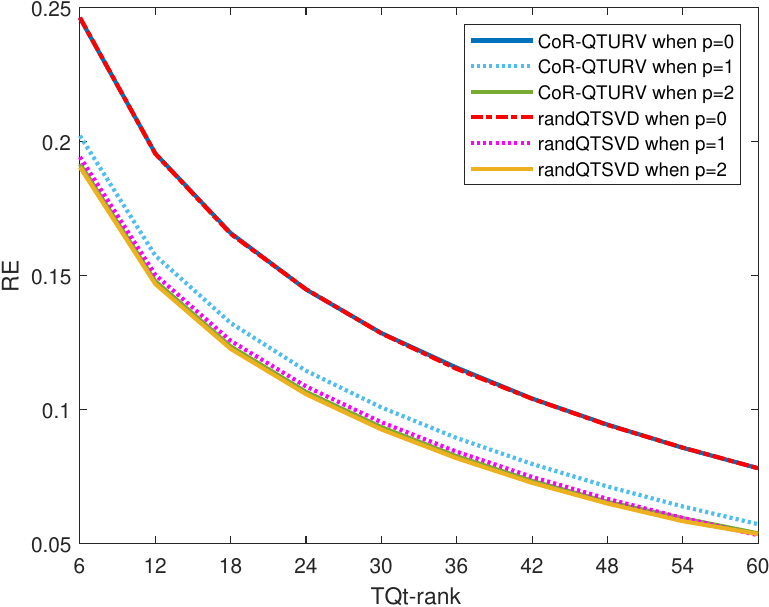}
		\end{minipage}
	}
	\subfigure{
		\begin{minipage}{6.cm}
			\centering
			\includegraphics[width = 6.cm]{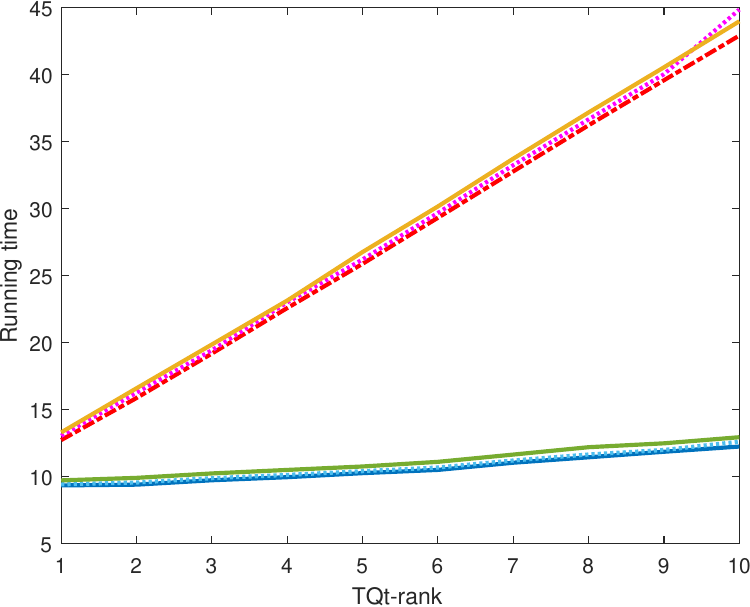}
		\end{minipage}
	}
	\caption{Comparison of  the RE and running time of implementing CoR-QTURV and randQTSVD to  Example \ref{eg6} with TQt-rank $K=6, 12, \cdots, 60.$}\label{f9}
\end{figure}

For video ``Landscape", let the TQt-rank $K=10:10:150$, the comparison of  CoR-QURV and randQSVD is displayed in Figure \ref{f10}.
\begin{figure}[htpb]	
	\centering
	\subfigure{
		\begin{minipage}{6.cm}
			\centering
			\includegraphics[width = 6.cm]{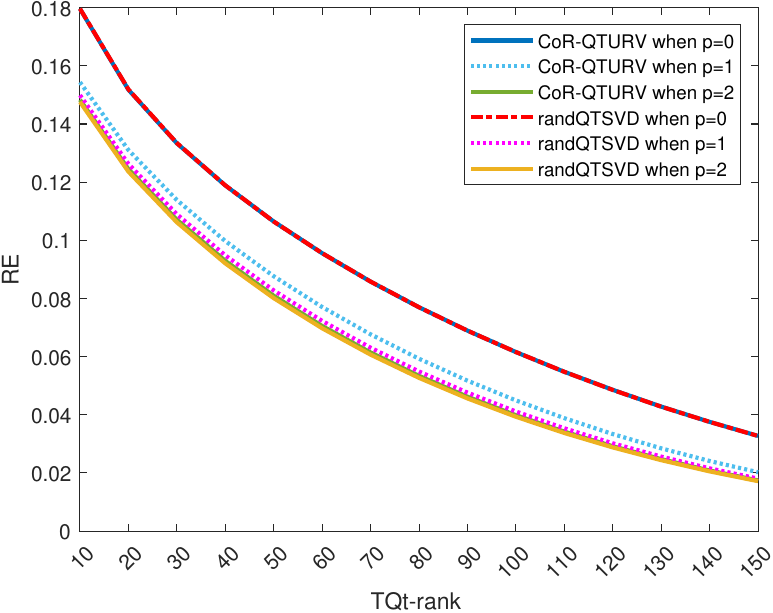}
			
		\end{minipage}
	}
	\subfigure{
		\begin{minipage}{6.cm}
			\centering
			\includegraphics[width = 6.cm]{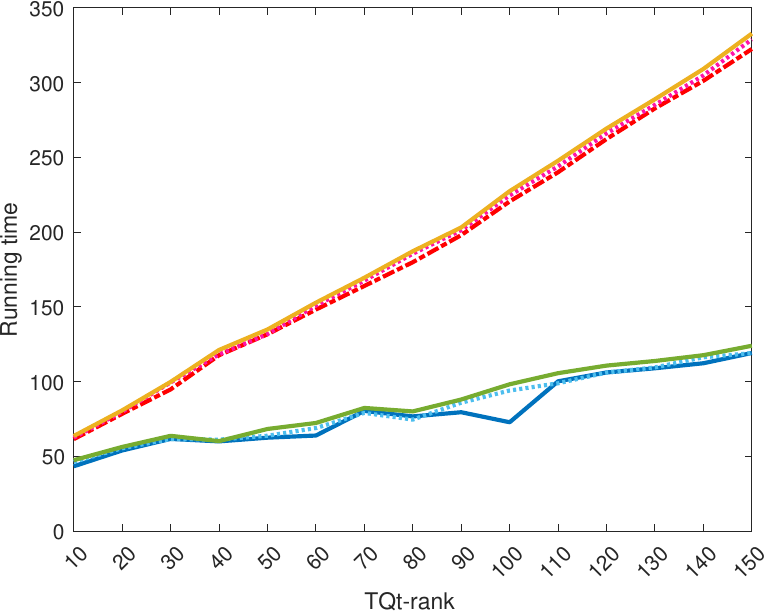}
		\end{minipage}
	}
	\caption{Comparison of  the RE and running time of implementing CoR-QTURV and randQTSVD to  Example \ref{eg6} with TQt-rank $K=10, 20, \cdots, 150.$}\label{f10}
\end{figure}

The RE decreases as the TQt-rank increases, as observed in Figure \ref{f9}-\ref{f10}. The techniques necessitate an increased amount of computation time as the TQt-rank increases. Additionally, the CoR-QTURV exhibits slightly inferior performance in terms of RE compared to randQTSVD when the same power parameter $p$ is used, as illustrated in Figure \ref{f9}-\ref{f10}. The RE are more poor when $p=0$ than when $p=1, 2$. The RE can accomplish the optimal outcome for randQTSVD and CoR-QTURV when the value of $p$ is equal to 2. For all values of $p$, CoR-QTURV outperforms randQTSVD in terms of running time.

Next, the visual results of video ``Football" are shown in Figure \ref{f11} (the results of the first frame) and Figure \ref{f12} (the results of the last frame). In this experiment, the power parameter $p=1$, and the TQt-rank $K=6, 12, 18, 36$. The first column is the original frame, the second column is the recovery of randQTSVD, and the last column is the recovery of  CoR-QTURV.
\begin{figure}[htpb]	
	\centering
	\subfigure[Original]{
		\begin{minipage}{4.cm}
			\centering
			\includegraphics[width = 4.cm]{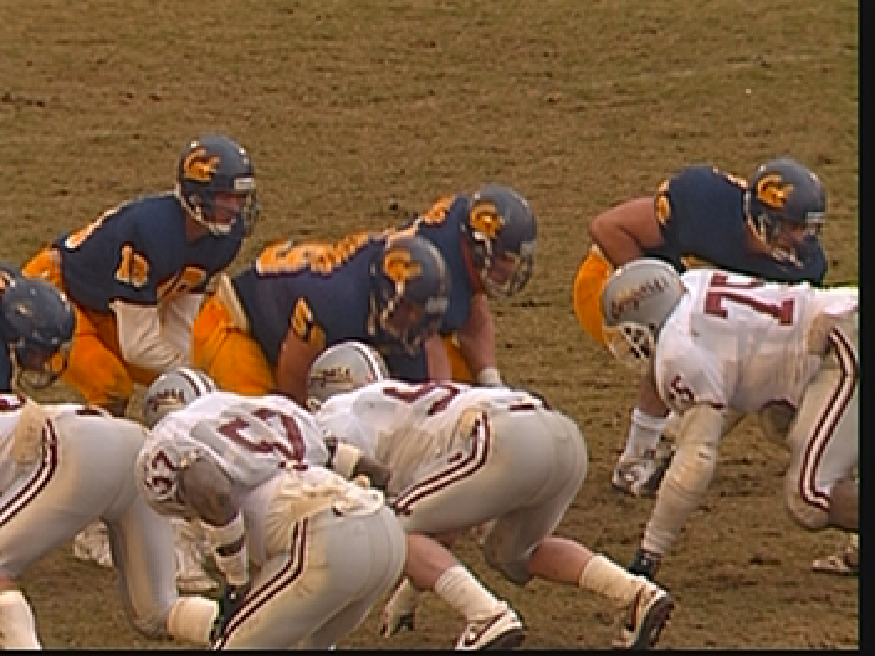}
			
		\end{minipage}
	}
	\subfigure[randQTSVD $K=6$]{
		\begin{minipage}{4.cm}
			\centering
			\includegraphics[width = 4cm]{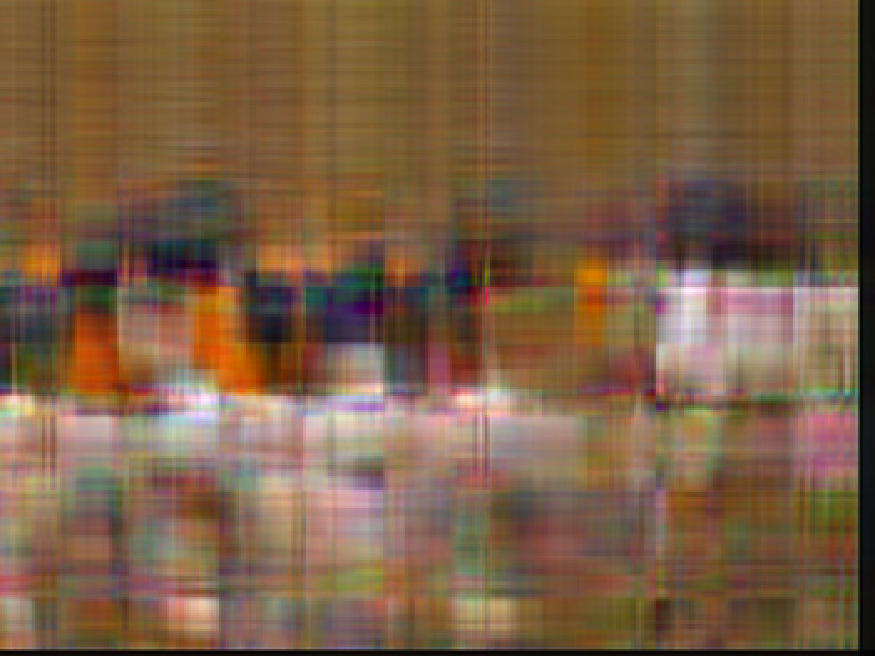}
			
		\end{minipage}
	}
	\subfigure[CoR-QTURV $K=6$]{
		\begin{minipage}{4.cm}
			\centering
			\includegraphics[width = 4.cm]{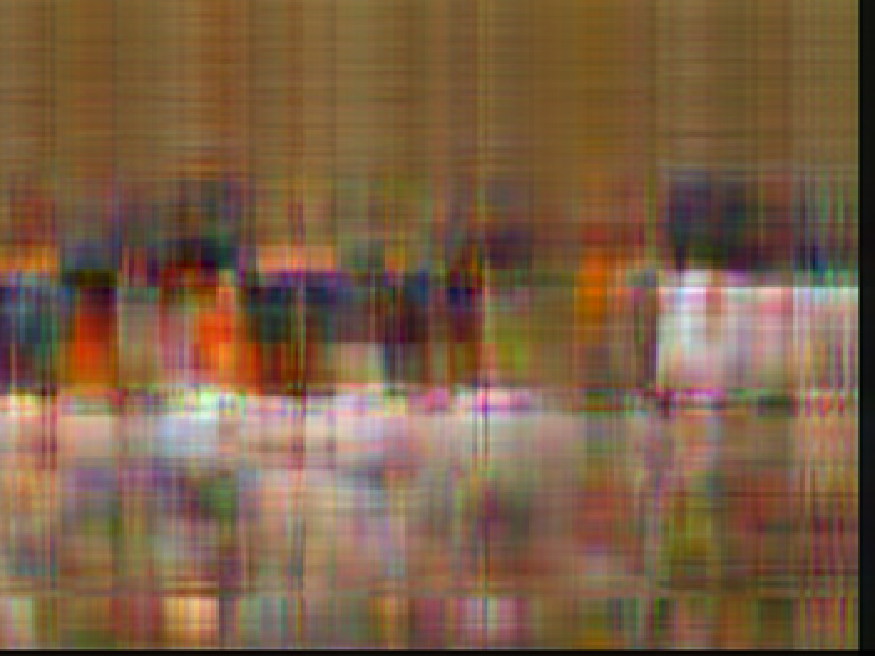}
		\end{minipage}
	}\\
	\subfigure[Original]{
		\begin{minipage}{4.cm}
			\centering
			\includegraphics[width = 4.cm]{ball1.png}
			
		\end{minipage}
	}
	\subfigure[randQTSVD $K=12$]{
		\begin{minipage}{4.cm}
			\centering
			\includegraphics[width = 4.cm]{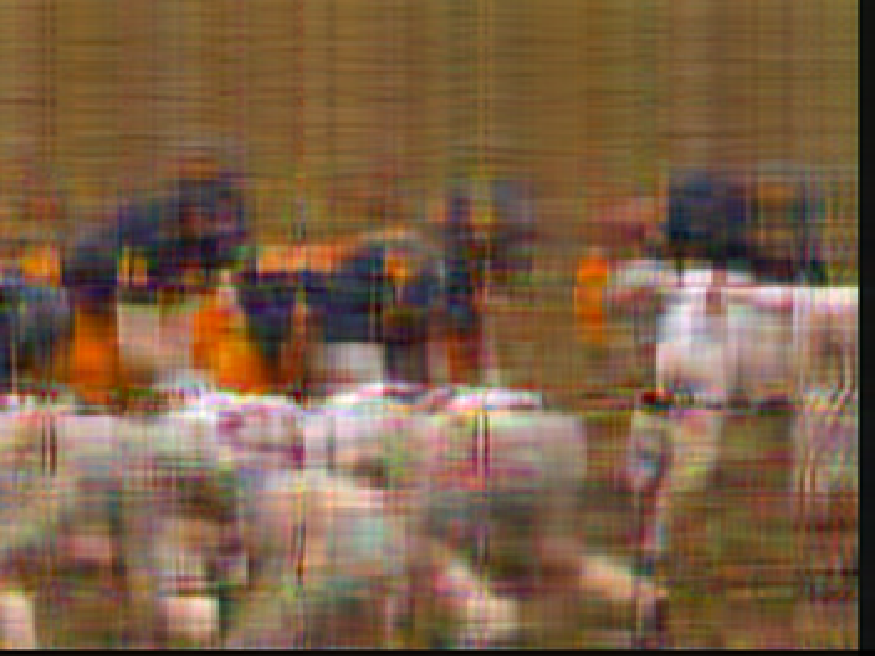}
			
		\end{minipage}
	}
	\subfigure[CoR-QURV $K=12$]{
		\begin{minipage}{4.cm}
			\centering
			\includegraphics[width = 4.cm]{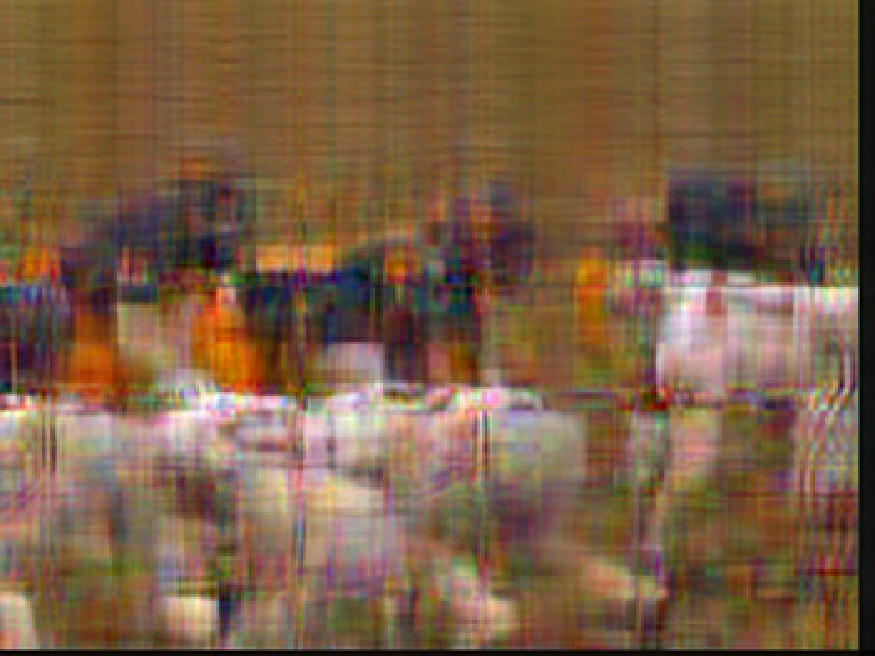}
		\end{minipage}
	}\\
	\subfigure[Original]{
		\begin{minipage}{4.cm}
			\centering
			\includegraphics[width = 4.cm]{ball1.png}
			
		\end{minipage}
	}
	\subfigure[randQSVD $K=18$]{
		\begin{minipage}{4.cm}
			\centering
			\includegraphics[width = 4.cm]{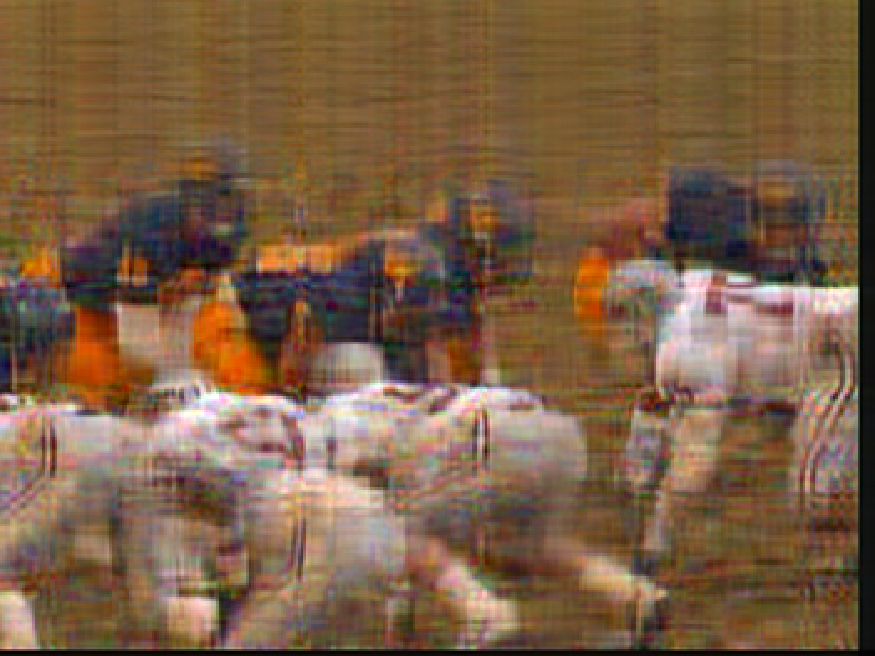}
			
		\end{minipage}
	}
	\subfigure[CoR-QURV $K=18$]{
		\begin{minipage}{4.cm}
			\centering
			\includegraphics[width = 4.cm]{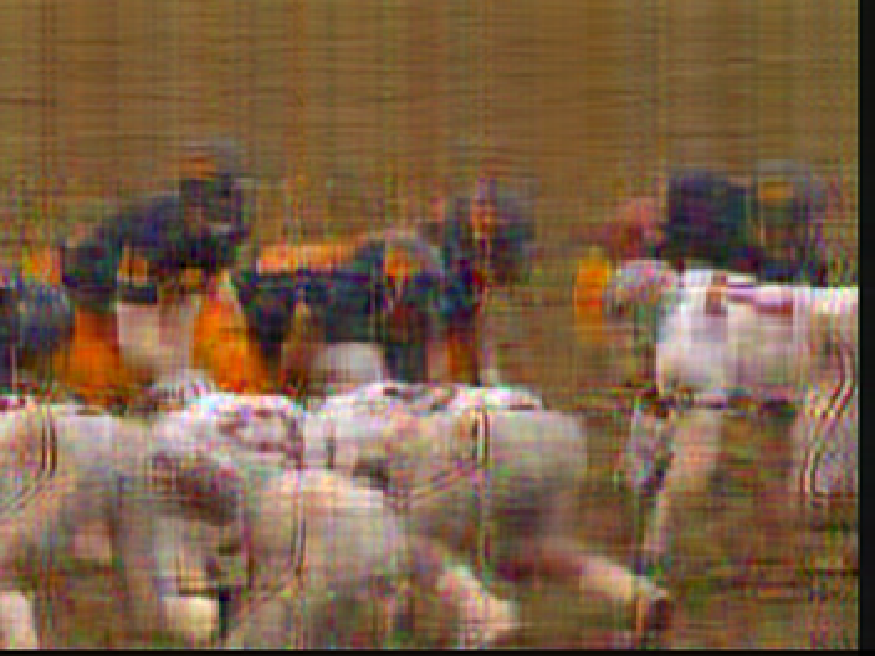}
		\end{minipage}
	}\\
	\subfigure[Original]{
		\begin{minipage}{4.cm}
			\centering
			\includegraphics[width = 4.cm]{ball1.png}
			
		\end{minipage}
	}
	\subfigure[randQSVD $K=36$]{
		\begin{minipage}{4.cm}
			\centering
			\includegraphics[width = 4.cm]{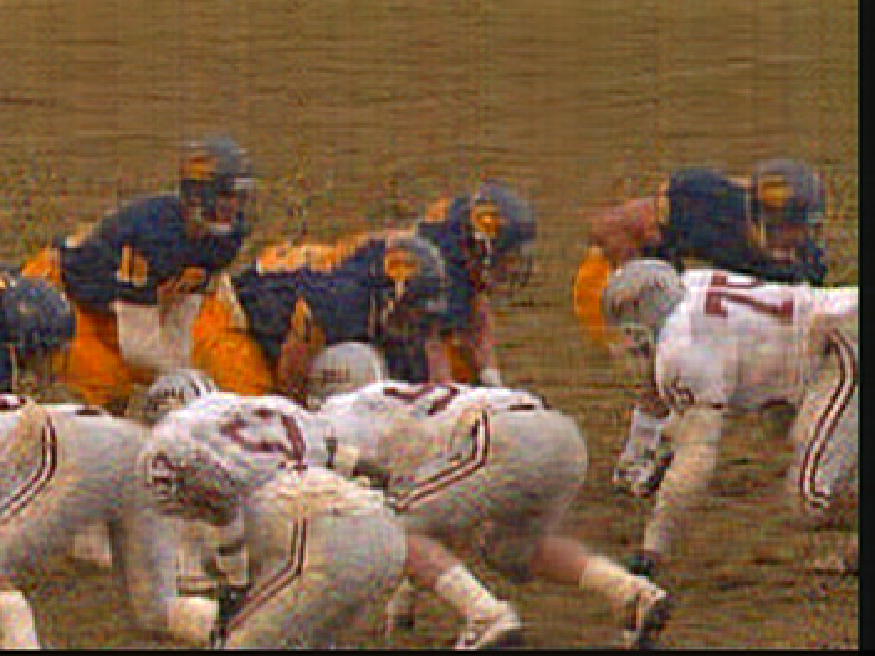}
			
		\end{minipage}
	}
	\subfigure[CoR-QURV $K=36$]{
		\begin{minipage}{4.cm}
			\centering
			\includegraphics[width = 4.cm]{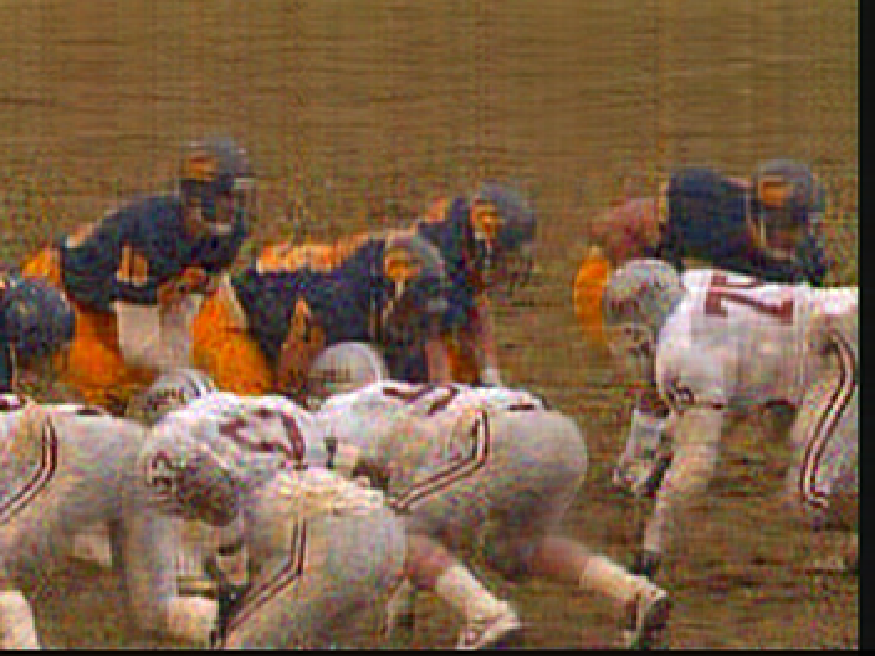}
		\end{minipage}
	}
	\caption{The TQt-rank-K approximation results of  ``Football" (the first frame) by utilizing CoR-QTURV and randQTSVD.}\label{f11}
\end{figure}

\begin{figure}[htpb]	
	\centering
	\subfigure[Original]{
		\begin{minipage}{4.cm}
			\centering
			\includegraphics[width = 4.cm]{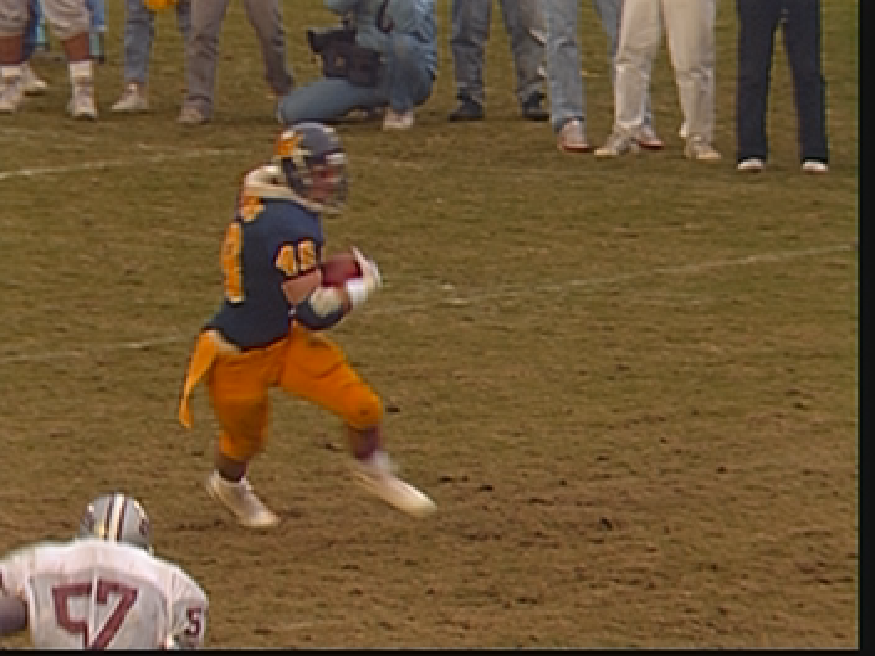}
			
		\end{minipage}
	}
	\subfigure[randQSVD $K=6$]{
		\begin{minipage}{4.cm}
			\centering
			\includegraphics[width = 4.cm]{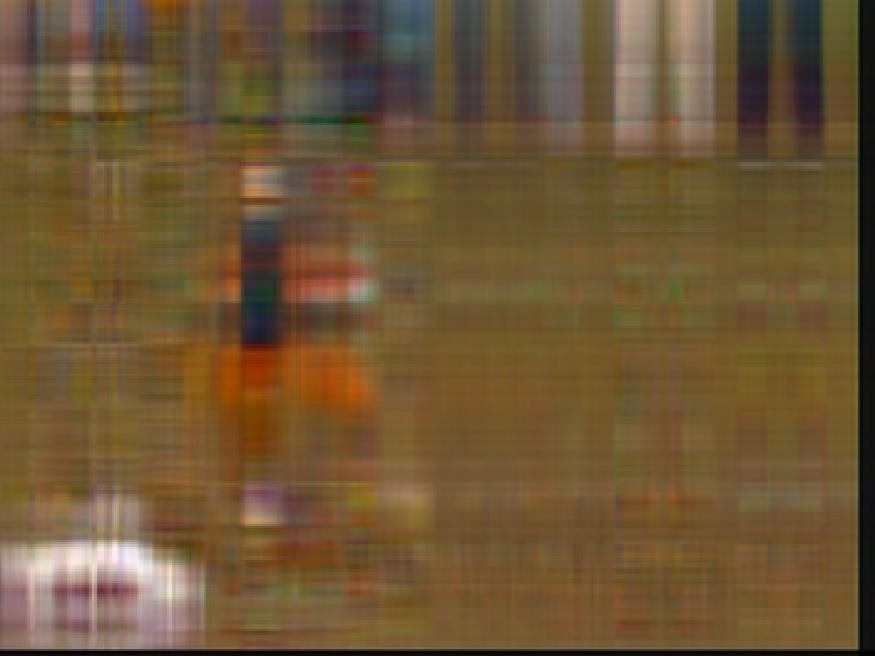}
			
		\end{minipage}
	}
	\subfigure[CoR-QURV $K=6$]{
		\begin{minipage}{4.cm}
			\centering
			\includegraphics[width = 4.cm]{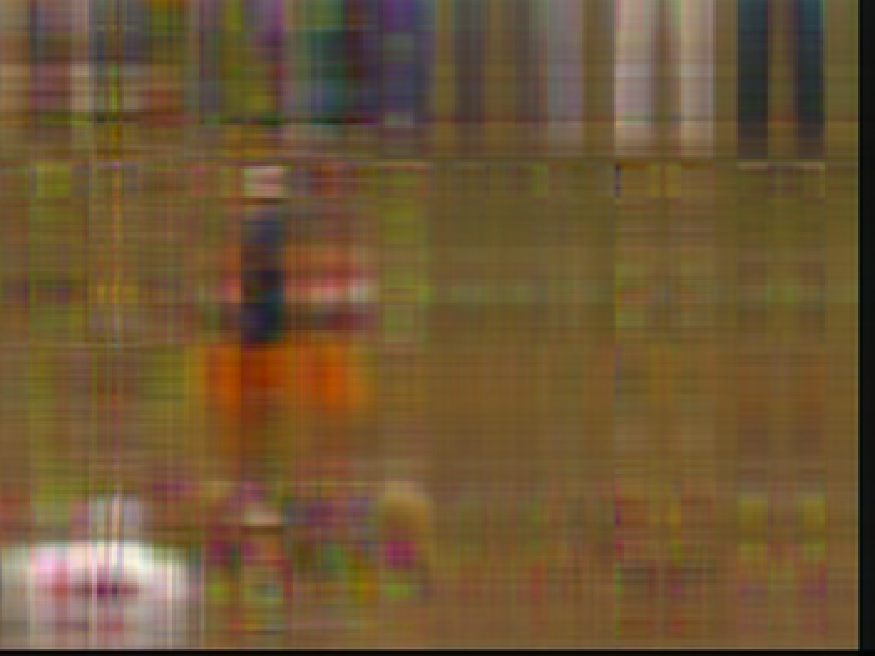}
		\end{minipage}
	}\\
	\subfigure[Original]{
		\begin{minipage}{4.cm}
			\centering
			\includegraphics[width = 4.cm]{ball125.png}
			
		\end{minipage}
	}
	\subfigure[randQSVD $K=12$]{
		\begin{minipage}{4.cm}
			\centering
			\includegraphics[width = 4.cm]{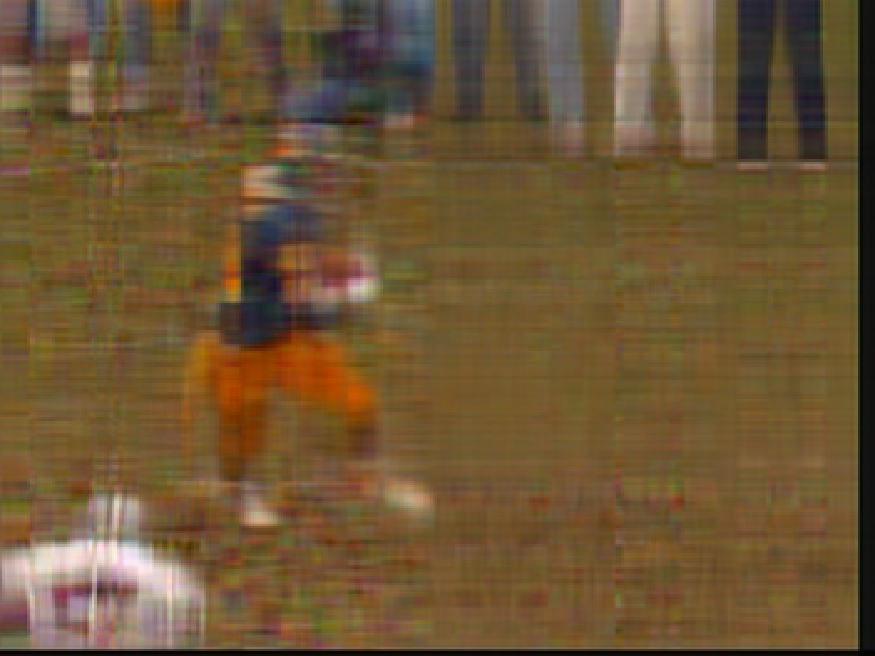}
			
		\end{minipage}
	}
	\subfigure[CoR-QURV $K=12$]{
		\begin{minipage}{4.cm}
			\centering
			\includegraphics[width = 4.cm]{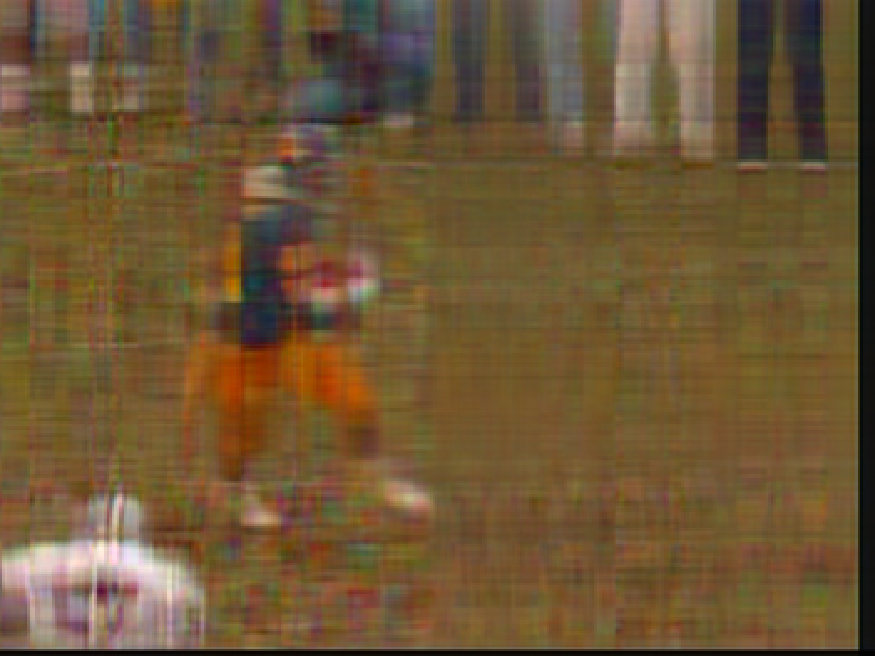}
		\end{minipage}
	}\\
	\subfigure[Original]{
		\begin{minipage}{4.cm}
			\centering
			\includegraphics[width = 4.cm]{ball125.png}
			
		\end{minipage}
	}
	\subfigure[randQSVD $K=18$]{
		\begin{minipage}{4.cm}
			\centering
			\includegraphics[width = 4.cm]{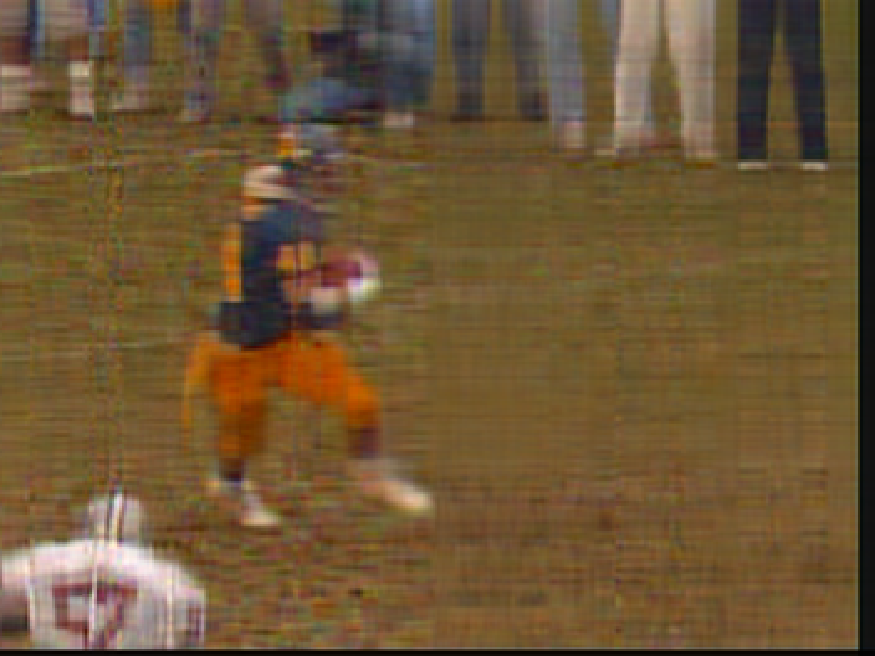}
			
		\end{minipage}
	}
	\subfigure[CoR-QURV $K=18$]{
		\begin{minipage}{4.cm}
			\centering
			\includegraphics[width = 4.cm]{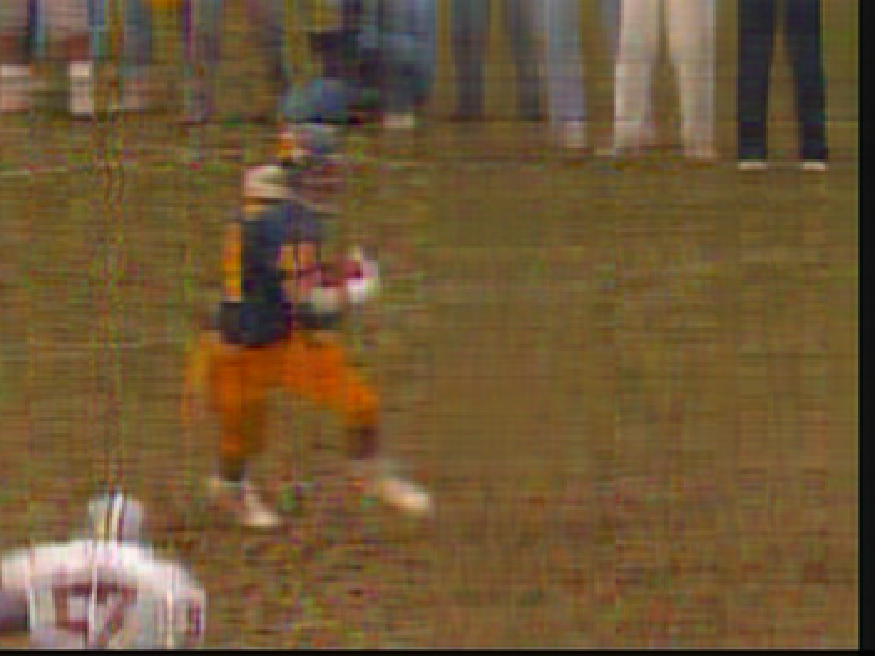}
		\end{minipage}
	}\\
	\subfigure[Original]{
		\begin{minipage}{4.cm}
			\centering
			\includegraphics[width = 4.cm]{ball125.png}
			
		\end{minipage}
	}
	\subfigure[randQSVD $K=36$]{
		\begin{minipage}{4.cm}
			\centering
			\includegraphics[width = 4.cm]{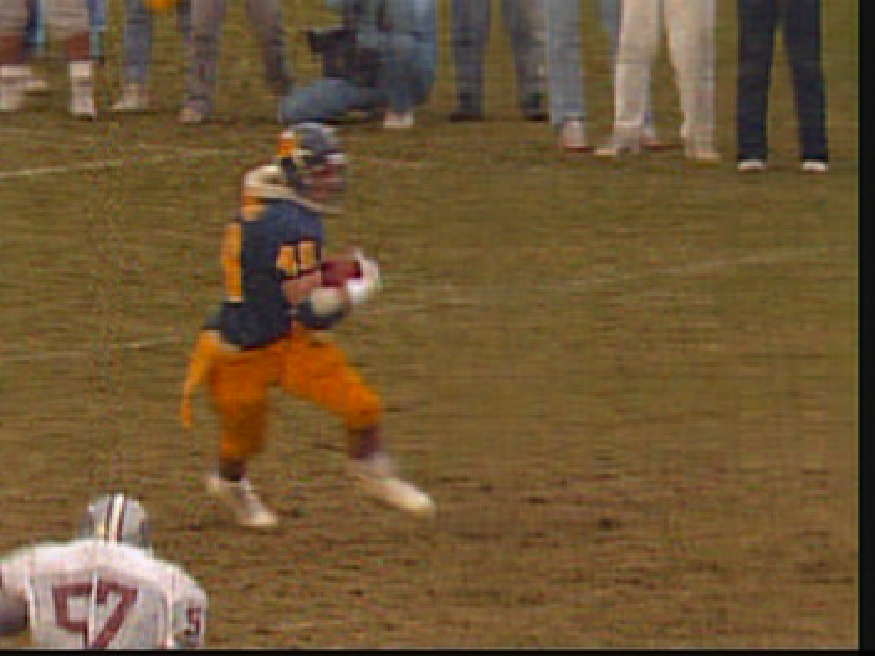}
			
		\end{minipage}
	}
	\subfigure[CoR-QURV $K=36$]{
		\begin{minipage}{4.cm}
			\centering
			\includegraphics[width = 4.cm]{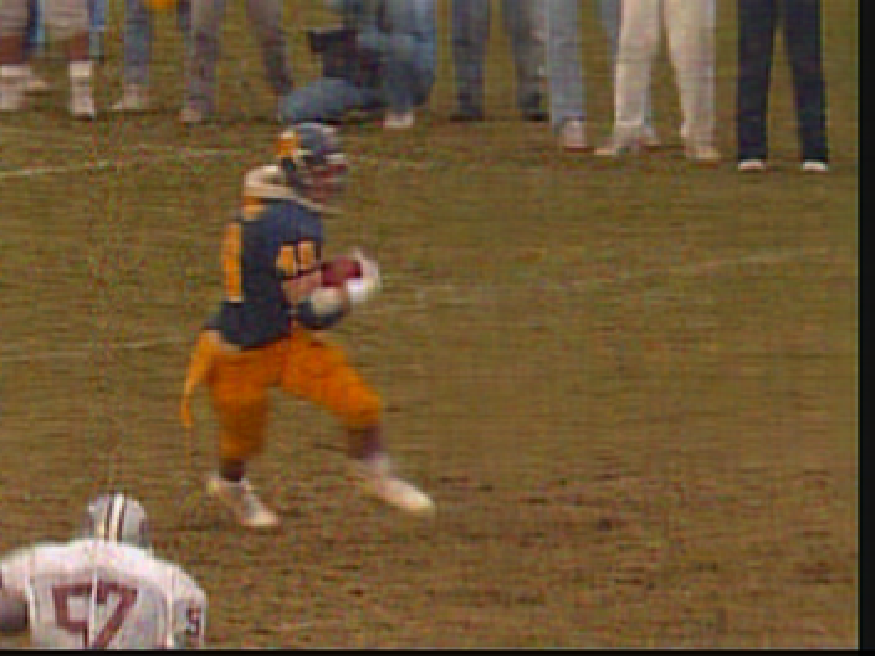}
		\end{minipage}
	}
	\caption{The TQt-rank-K approximation results of ``Football" (the last frame) by utilizing CoR-QTURV and randQTSVD. }\label{f12}
\end{figure}

The numbers on jerseys in the frame are roughly visible, and the visual impact of the original frame and the recovered frame is minimal when $K=36$, as evidenced by the visual results in Figure \ref{f11}-\ref{f12}.  The original video can be approximately restored by both CoR-QTURV and randQTSVD algorithms when the TQt-rank is small, and there is not a significant visual difference between them.
\section{Conclusion}
\label{C}
This paper proposes a novel decomposition for quaternion matrix and quaternion tensor. Firstly, the definition of QTUTV is provided, followed by the definition of this decomposition for quaternion tensor. In order to enhance the efficiency of the algorithms, a randomized technique is employed, resulting in the development of the CoR-QURV and CoR-QTURV. The error bound analysis is presented for randomized approaches. Finally, the experiment conducted on synthetic data, color images, and color videos provides evidence of the effectiveness of the proposed algorithms. Additionally, the developed quaternion-based decomposition can be used for the low-rank quaternion tensor completion model or quaternion tensor robust component analysis problem. The quaternion-based UTV decomposition can be considered as an alternative to the truncated QSVD in certain problems. 
\appendix
\section{Proof of  Theorem \ref{p1}}
To prove Proposition 1, we first present several key results for using later on.
\begin{proposition}	\cite{DBLP:journals/siamrev/HalkoMT11}
	Let $h$ be a real valued Lipschitz function on matrices, then for all $\mathbf{X}$:
	\begin{equation}
		\mid h(\mathbf{X})-h(\mathbf{Y})\mid\leq L\|\mathbf{X}-\mathbf{Y}\|_F, \quad 
	\end{equation}
	where $L>0$. For any $m\times n$ standard Gaussian matrix $\mathbf{G}$ and $u>0$, then $\mathbb{P}\{h(\mathbf{G})\geq\mathbb{E}h(\mathbf{G})+Lu\}\leq e^{-u^2/2}$.
	\label{p3}
\end{proposition}
\begin{proposition}\cite{DBLP:journals/siamsc/LiuLJ22}
	For  $\dot{\mathbf{A}}=\mathbf{A}_0+\mathbf{A}_1i+\mathbf{A}_2j+\mathbf{A}_3k\in\mathbb{H}^{M\times N}$, define the real counterpart and the column representation respectively as follows\\ $\gamma_{\dot{\mathbf{A}}}=\left[\begin{array}{cccc}
		\mathbf{A}_0& -\mathbf{A}_1&-\mathbf{A}_2&-\mathbf{A}_3 \\
		\mathbf{A}_1& \mathbf{A}_0&-\mathbf{A}_3&\mathbf{A}_2\\
		\mathbf{A}_2& \mathbf{A}_3&\mathbf{A}_0&-\mathbf{A}_1\\
		\mathbf{A}_3& -\mathbf{A}_2&\mathbf{A}_1&\mathbf{A}_0\end{array}\right]$, \qquad $\dot{\mathbf{A}}_C=\left[ \begin{array}{cccc}\mathbf{A}_0\\ \mathbf{A}_1\\ \mathbf{A}_2\\ \mathbf{A}_3\end{array}\right]$. \\Then $\|\dot{\mathbf{A}}\|_F=\frac{1}{2}\|\gamma_{\dot{\mathbf{A}}}\|_F$ and $\|\dot{\mathbf{A}}\|_2=\|\dot{\mathbf{A}}_C\|_2$.
	\label{p4}
\end{proposition}
\begin{proposition}
	For $\dot{\mathbf{\Omega}}=\mathbf{\Omega}_0+\mathbf{\Omega}_1i+\mathbf{\Omega_2}j+\mathbf{\Omega_3}k$, where $\mathbf{\Omega}_i, i=0, 1, 2,3$ are random and independently Gaussian random matrices. Define a function $h(\dot{\mathbf{\Omega}})=\|\dot{\mathbf{\Omega}}\|_2$. Then by utilizing Lemma 6 in \cite{DBLP:journals/siamsc/LiuLJ22}, we have $\mathbb{E}(h(\dot{\mathbf{\Omega}})\leq3(\sqrt{M}+\sqrt{N})<3(\sqrt{M}+\sqrt{N})+1\triangleq\epsilon$.
	\label{pr5}
\end{proposition}

\begin{proposition}\cite{DBLP:journals/siamsc/Gu15}
	Let $g(\cdot)$be a nonnegative continuously differentiable function with $g(0)=0$, and $\mathbf{G}$ is a random matrix, then
	$ \mathbb{E}_g(\mathbf{G})=\int_0^\infty g^{'}(x)\mathbb{P}\{\|\mathbf{G}\|_2\geq x\}dx$.
	\label{p5}
\end{proposition}
\begin{proof}
	\text{Define }$	h(\dot{\mathbf{A}}_C)=\|\gamma_{\dot{\mathbf{A}}}\|_2=\|[\mathbf{J}_0\dot{\mathbf{A}}_C \;\mathbf{J}_1\dot{\mathbf{A}}_C \; \mathbf{J}_2\dot{\mathbf{A}}_C \; \mathbf{J}_3\dot{\mathbf{A}}_C]\|_2$,
	 where $\mathbf{J}_0=\mathbf{I}_4\otimes\mathbf{I}_M$, $\mathbf{J}_1=[-\mathbf{e}_2\; \mathbf{e}_1 \;\mathbf{e}_4\;-\mathbf{e}_3]\otimes\mathbf{I}_M$, $\mathbf{J}_2=[-\mathbf{e}_3\; -\mathbf{e}_4 \;\mathbf{e}_1\;\mathbf{e}_2]\otimes\mathbf{I}_M$, $\mathbf{J}_3=[-\mathbf{e}_4\; \mathbf{e}_3 \;-\mathbf{e}_2\;\mathbf{e}_1]\otimes\mathbf{I}_M$, with $\mathbf{e}_i$ be the $i$th column of $\mathbf{I}_4$. Utilizing Proposition \ref{p3} and Proposition \ref{p4}, we have
	 \begin{equation*}
	 h(\dot{\mathbf{A}}_C-\dot{\mathbf{B}}_C)=\mid\|\gamma_{\dot{\mathbf{A}}}\|_2-\|\gamma_{\dot{\mathbf{B}}}\|_2\mid=\|\dot{\mathbf{A}}\|_2-\|\dot{\mathbf{B}}\|_2\leq\|\dot{\mathbf{A}}-\dot{\mathbf{B}}\|_2\leq\|\dot{\mathbf{A}}_C-\dot{\mathbf{B}}_C\|_F.
	 \end{equation*}
  That means that $h(\cdot)$ is a Lipschitz function. Next, utilizing Proposition \ref{pr5} with Lipschitz constant $L=1$, we have $\mathbb{P}\{\|\dot{\mathbf{G}}\|_2\geq x\}\leq e^{-u^2/2}$, where $u=x-\epsilon$.
	Then, we first define the real function $g(x)\triangleq\sqrt{\frac{\alpha^2x^2}{1+\beta^2x^2}}$ with $g^{'}(x)=\frac{\alpha^2x^2}{(1+\beta^2x^2)^2\sqrt{\frac{\alpha^2x^2}{1+\beta^2x^2}}}$ , where $\alpha>0$, $\beta>0$. Then we have 
	\begin{equation}\label{a1}
		\begin{split}
			\mathbb{E}\left(\sqrt{\frac{\alpha^2\|\dot{\mathbf{\Omega}}\|_2^2}{1+\beta^2\|\dot{\mathbf{\Omega}}\|_2^2}}\right)&=\mathbb{E}\left(\sqrt{\frac{\alpha^2\|\dot{\mathbf{\Omega}}_C\|_2^2}{1+\beta^2\|\dot{\mathbf{\Omega}}_C\|_2^2}}\right)\\&=\int_0^\infty g^{'}(x)\mathbb{P}\{\|\dot{\mathbf{\Omega}}\|_2\geq x\}dx\\&\leq\sqrt{\frac{\alpha^2\epsilon^2}{1+\beta^2\epsilon^2}}+\int_\epsilon^\infty \frac{\alpha^2x^2}{(1+\beta^2x^2)^2\sqrt{\frac{\alpha^2x^2}{1+\beta^2x^2}}}e^{-\frac{(x-\epsilon)^2}{2}}dx\\&\leq\sqrt{\frac{\alpha^2\epsilon^2}{1+\beta^2\epsilon^2}}+\frac{\alpha^2}{(1+\beta^2\epsilon^2)^2\sqrt{\frac{\alpha^2\epsilon^2}{1+\beta^2\epsilon^2}}}\int_0^\infty (u+\epsilon)e^{-\frac{(u)^2}{2}}du\\&=\sqrt{\frac{\alpha^2\epsilon^2}{1+\beta^2\epsilon^2}}+\frac{\alpha^2}{(1+\beta^2\epsilon^2)^2\sqrt{\frac{\alpha^2\epsilon^2}{1+\beta^2\epsilon^2}}}(1+\epsilon\sqrt{\frac{\pi}{2}}).
		\end{split}
	\end{equation}
	Comparing the inequality in Proposition \ref{p1} with \eqref{a1}, we need to find a constant $v>0$ such that 
	\begin{equation}\label{a2}
		\sqrt{\frac{\alpha^2\epsilon^2}{1+\beta^2\epsilon^2}}+\frac{\alpha^2}{(1+\beta^2\epsilon^2)^2\sqrt{\frac{\alpha^2\epsilon^2}{1+\beta^2\epsilon^2}}}(1+\epsilon\sqrt{\frac{\pi}{2}})\leq\sqrt{\frac{\alpha^2v^2}{1+\beta^2v^2}},
	\end{equation} 
	which leads to $v^2-\epsilon^2\geq(\epsilon\sqrt{\frac{\pi}{2}}+1)\frac{1+\beta^2v^2}{1+\beta^2\epsilon^2}[\frac{\sqrt{\frac{\alpha^2\epsilon^2}{1+\beta^2\epsilon^2}}}{\sqrt{\frac{\alpha^2v^2}{1+\beta^2v^2}}}+1]$. 
	
	The right side of the inequality approaches maximum value as  $\beta\longrightarrow\infty$. Thus 
	$v^2-\epsilon^2\geq\frac{2c^2}{\epsilon^2}(\epsilon\frac{\pi}{2}+1)$, which solves to $v\geq\frac{\epsilon^2}{\sqrt{\epsilon^2-2(\epsilon\sqrt{\frac{\pi}{2}}+1)}}$. Because the defined $\epsilon\geq7$, the inequality holds when $v=3(\sqrt{M}+\sqrt{N})+4=\epsilon+3$.
\end{proof}
\section{Proof of  Proposition \ref{p2}}

\begin{proof}
	Based on the Theorem 7 in \cite{DBLP:journals/siamsc/LiuLJ22}, we have $\mathbb{P}\{\|\dot{\mathbf{G^\dagger}}\|_2\geq x\}\leq \bar{c}x^{-2(2N-2M+2)}$, where  $\bar{c}=\frac{\breve{c}}{4^{2N-2M+2}}$ and $\breve{c}=\frac{\pi^{-3}}{4(N-M+1)(2N-2M+3)}\cdot[\frac{e\sqrt{4N+2}}{2N-2M+2}]^{2(2N-2M+2)}$.
	
	Following a similar track in the proof of Proposition \ref{p1}, we have 
	\begin{equation}\label{a3}
		\begin{split}
			\mathbb{E}\left(\sqrt{\frac{\alpha^2\|\dot{\mathbf{\Omega}}^\dagger\|_2^2}{1+\beta^2\|\dot{\mathbf{\Omega}}^\dagger\|_2^2}}\right)&=\mathbb{E}\left(\sqrt{\frac{\alpha^2\|\dot{\mathbf{\Omega}}^{\dagger}_C\|_2^2}{1+\beta^2\|\dot{\mathbf{\Omega}}^{\dagger}_C\|_2^2}}\right)\\&=\int_0^\infty g^{'}(x)\mathbb{P}\{\|\dot{\mathbf{\Omega}}^{\dagger}\|_2\geq x\}dx\\&\leq\sqrt{\frac{\alpha^2c^2}{1+\beta^2c^2}}+\int_c^\infty \frac{\alpha^2x^2}{(1+\beta^2x^2)^2\sqrt{\frac{\alpha^2x^2}{1+\beta^2x^2}}}\bar{c}x^{-2(2N-2M+2)}dx\\&\leq\sqrt{\frac{\alpha^2c^2}{1+\beta^2c^2}}+\frac{\alpha^2\bar{c}}{(1+\beta^2c^2)^2\sqrt{\frac{\alpha^2c^2}{1+\beta^2c^2}}}\int_c^\infty x^{-4(N-M+1)+1}dx\\&=\sqrt{\frac{\alpha^2c^2}{1+\beta^2c^2}}+\frac{\alpha^2\bar{c}}{(1+\beta^2c^2)^2\sqrt{\frac{\alpha^2c^2}{1+\beta^2c^2}}}\cdot\frac{c^{-4(N-M+1)+2}}{4(N-M+1)-2}.
		\end{split}
	\end{equation}
	Let $N-M=q\geq0$, then we need to find a $v>0$ such that
	\begin{equation*}
	\sqrt{\frac{\alpha^2c^2}{1+\beta^2c^2}}+\frac{\alpha^2\bar{c}}{(1+\beta^2c^2)^2\sqrt{\frac{\alpha^2c^2}{1+\beta^2c^2}}}\cdot\frac{c^{-4(q+1)}}{4(q+1)-2}\leq\sqrt{\frac{\alpha^2v^2}{1+\beta^2v^2}}.
	\end{equation*} This equation is similar to equation \eqref{a2}, with difference being the coefficients in the second term on the left side. Thus, its solution similarly satisfies $v\geq\frac{c^2}{\sqrt{c^2-\frac{2\bar{c}c^2c^{-4(q+1)}}{4(q+1)-2}}}$. Next, substituting the value of $\bar{c}$ to the right side of the inequality, we have
\begin{equation*}
	v\geq\frac{c}{\sqrt{1-\frac{\pi^{-3}}{4(q+1)(2q+1)(2q+3)}(\frac{4(q+1)}{e\sqrt{4n+2}}c)^{-4(q+1)}}}.
\end{equation*} 
	The value $v=\frac{e\sqrt{4N+2}}{q+1}$ satisfies this inequality for $c=\frac{e\sqrt{4N+2}}{4(q+1)}(\frac{\pi^{-3}}{4(2q+1)(q+1)})^{\frac{1}{4(q+1)}}$.
\end{proof}

\section*{Acknowledgments}
We would like to acknowledge the assistance of volunteers in putting
together this example manuscript and supplement.

\bibliographystyle{siamplain}
\bibliography{references}
\end{document}